\documentclass[10pt,letter,reqno]{amsart}
\usepackage{amssymb,amsthm,bbm,graphicx,placeins,enumitem}
\usepackage{amstext}
\usepackage{natbib}
\usepackage{setspace,thmtools,thm-restate}
\usepackage[left=1.25in,right=1.25in,top=1in,bottom=1in]{geometry} 
\usepackage{float}
\usepackage{color}
\usepackage{hyperref}
\usepackage{mathtools}

\hypersetup{colorlinks,
linkcolor=blue,
filecolor=blue,
urlcolor=blue,
citecolor=blue}
\usepackage{changepage}

\usepackage{booktabs}
\usepackage{array}
\usepackage{multirow}
\usepackage[font=small]{caption}
\usepackage{subcaption}

\usepackage[foot]{amsaddr}

\allowdisplaybreaks

\newcommand{\widebar}[1]{\mkern 1.5mu\overline{\mkern-1.5mu#1\mkern-1.5mu}\mkern 1.5mu}

\DeclareMathOperator*{\argmax}{arg\,max}
 
\newtheorem{thm}{Theorem}

\theoremstyle{definition}
\newtheorem{assumption}{Assumption}

\def\Rmax{R_{\text{\textnormal{max}}}}
\def\Lmax{L_{\text{\textnormal{max}}}}
\def\Vmax{V_{\text{\textnormal{max}}}}
\def\bmax{b_{\text{\textnormal{max}}}}
\def\bmin{b_{\text{\textnormal{min}}}}

\def\mless{\preccurlyeq}
\mathtoolsset{centercolon}

\newcommand{\indicate}[1]{\mathbf{1}_{\{#1\}}}

\begin{document}
\title[Optimal Hour--Ahead Bidding in the Real--Time Electricity Market]{Optimal hour--ahead bidding in the real--time electricity market with battery storage using Approximate Dynamic Programming}


\begin{abstract}
There is growing interest in the use of grid--level storage to smooth variations in supply that are likely to arise with increased use of wind and solar energy.  Energy arbitrage, the process of buying, storing, and selling electricity to exploit variations in electricity spot prices, is becoming an important way of paying for expensive investments into grid--level storage.  Independent system operators such as the NYISO (New York Independent System Operator) require that battery storage operators place bids into an hour--ahead market (although settlements may occur in increments as small as 5 minutes, which is considered near ``real--time'').  The operator has to place these bids without knowing the energy level in the battery at the beginning of the hour, while simultaneously accounting for the value of leftover energy at the end of the hour.  The problem is formulated as a dynamic program. We describe and employ a convergent approximate dynamic programming (ADP) algorithm that exploits monotonicity of the value function to find a revenue--generating bidding policy; using optimal benchmarks, we empirically show the computational benefits of the algorithm. Furthermore, we propose a distribution--free variant of the ADP algorithm that does not require any knowledge of the distribution of the price process (and makes no assumptions regarding a specific real--time price model). We demonstrate that a policy trained on historical real--time price data from the NYISO using this distribution--free approach is indeed effective.
\end{abstract}

\author[Jiang and Powell]{Daniel R. Jiang$^*$ and Warren B. Powell$^*$}
\address{$^*$Department of Operations Research and Financial Engineering, Princeton University}
\maketitle

\section{Introduction}
Bidding into the electricity market can be a complicated process, mainly due to the requirement of balancing supply and demand at each point in the grid. To solve this issue, the Independent System Operators (ISOs) and the Regional Transmission Organizations (RTOs) generally use multi--settlement markets: several tiers of markets covering planning horizons that range from \emph{day--ahead} to \emph{real--time}. The idea is that the markets further away from the operating time settle the majority of the generation needed to handle the predicted load, while the markets closer to the operating time correct for the small, yet unpredictable deviations that may be caused by issues like weather, transmission problems, and generation outages (see, for example, \cite{Shahidehpour2002}, \cite{Eydeland2003}, \cite{Harris2006}, for more details). Settlements in these real--time markets are based on a set of intra--hour prices, typically computed at 5, 10, or 15 minute intervals, depending on the specific market in question. A settlement refers to the financial transaction after a generator \emph{clears the market}, which refers to being selected to either buy or sell energy from the market. If a generator does not clear the market, it remains idle and no settlement occurs. We refer to this situation as being \emph{out of the market}. 


Many ISO's and RTO's, such as the Pennsylvania--New Jersey--Maryland Interconnection (PJM), deal with the balancing market primarily through the day--ahead market. PJM's balancing market clears every 5 minutes (considered to be near ``real--time''), but the bids are all placed the previous day. See \cite{Eydeland2003} and the PJM Energy and Ancillary Services Market Operations Manual for more information.  In certain markets, however, it is not only possible to settle in real--time, but market participants can also submit bids each hour, for an hour in the future.  Thus, a bid (consisting of buy and sell prices) can be made at 1pm that will govern the battery between 2pm and 3pm. The process of both bidding and settling in real--time is a characteristic of the New York Independent System Operator (NYISO) real--time market and is the motivating example for this paper. Other prominent examples of markets that include a real--time bidding aspect include California ISO (CAISO) and Midcontinent ISO (MISO). In particular, our goal is to pair battery storage with hour--ahead bidding in the real--time market for revenue maximization, a strategy sometimes referred to as \emph{energy arbitrage}.

It is unlikely that profits from battery/energy arbitrage alone can be sustainable for a company; however, if performed optimally, it can be an important part of a range of profit generating activities (one such example is the frequency regulation market). See \cite{Walawalkar2007} for an economic analysis of using a storage device for both energy arbitrage (using a simple ``charge--off--peak and discharge--on--peak'' policy) and frequency regulation in the New York area. The analysis shows that in New York City (but not the surrounding areas), there is a ``high probability of positive NPV [net present value] for both energy arbitrage and regulation,'' but even so, there is still significant risk in not being able to recover the initial capital cost. However, the potential for more efficient and cost--effective technology combined with better control policies can make energy arbitrage feasible in the near future. Other studies on the topic of the value of storage include \cite{Sioshansi2009}, \cite{Sioshansi2011a}, and \cite{Byrne2012}.

In our problem, we assume that the goal is to optimally control a 1 MW battery; in practice, a company may operate a fleet of such batteries. Market rules state that we must bid in integer increments, meaning the possible actions at each settlement are to charge, discharge (both at a rate of 1 MW), or do nothing. Hence, our precise problem is to optimize the placement of two hour--ahead bids, a ``positive'' bid (for a quantity of $+1$ MW) and a ``negative'' bid (for a quantity of $-1$ MW) that correspond to selling (generation) and buying (negative generation), respectively, over a period of time such that purchased energy can be stored in the finite capacity battery. The goal is to maximize expected revenue. Further, given that our model is tailored to battery storage (inherently small capacity), it is reasonable to assume no price impact (i.e., our bids do not affect the spot prices of electricity). In the real--time market, bidding for the operating hour closes an hour in advance and the hour--ahead bid is fixed for the entire operating hour.


This paper makes the following contributions. We describe, in detail, a mathematical model of the bidding process in the real--time electricity market and formulate the sequential decision problem as a Markov Decision Process (MDP). Along the way, we show the structural properties of the problem (monotonicity of the contribution and value functions) that we utilize in our solution technique. Next, we describe and benchmark a convergent approximate dynamic programming algorithm called \emph{Monotone--ADP} (M--ADP) (\cite{Jiang2013}) that can be used to obtain an approximate, but near--optimal bidding policy. We also present a new version of Monotone--ADP utilizing post--decision states that allows us to train bidding policies without any model or knowledge of the distribution of real--time prices (which we call a distribution--free method), allowing our solution technique to be easily adopted in practice. Finally, we present a case study detailing the results of an ADP policy trained using only historical real--time price data from the NYISO. In the case study, we also compare the ADP policy to other rule--based policies, two of which are from the energy arbitrage literature and one from our industry contacts. All proofs can be found in the Appendix.
%

\section{Literature Review}
With renewable energy sources like wind and solar becoming more established, the problem of energy storage is also becoming increasingly important. In this section, we first review studies dedicated solely to storage and then move on to those that consider the bidding aspect. Lastly, we discuss algorithmic techniques similar to our proposed method (Monotone--ADP).

Coupling wind energy with storage has been well--studied in a variety of ways. The paper by \cite{Kim2011} poses a wind energy commitment problem given storage and then analytically determines the optimal policy for the infinite horizon case. \cite{Sioshansi2011} uses ideas from economics and game theory (i.e., the Stackelberg Model) to make several conclusions, including the finding that the value of storage increases with market--competitiveness. In addition, \cite{Greenblatt2007} finds that for high green house gas (GHG) emissions prices, compressed air energy storage is a better choice as a supplemental generator to wind energy when compared to natural gas turbines. The well--known smoothing effects of energy storage on intermittent renewable sources is studied in the context of wind power output by \cite{Paatero2005}.

Another problem within this realm is the storage of natural gas, which involves optimally controlling injection and withdrawal of gas from a storage facility that is typically underground. \cite{Carmona2010} uses a technique known as \emph{optimal switching} to solve a natural gas storage problem; computationally, the value function is approximated using basis functions. In a similar vein, \cite{Thompson2009} formulates a stochastic control problem and numerically solve the resulting integro--differential equation to arrive at the optimal policy. \cite{Lai2010} proposes using an ADP algorithm along with an approximation technique to reduce the number of state space dimensions for natural gas storage valuation. 

Other energy storage problems include reservoir management (see \cite{Nandalal2007}) and pairing solar with battery storage (see \cite{Barnhart2013}). It quickly becomes clear that all of these problems are similar; in fact, \cite{Secomandi2010} gives the structure of the optimal policy for trading generic commodities given storage. At its core, energy storage has similarities to an array of classical problems related to operations research, such as resource allocation and inventory control.

There are also many studies that consider the bidding aspect of the electricity markets. One significant difference between many of these studies and our paper is that, rather than placing many bids at once, we consider a sequential, hourly bidding problem.  \cite{Lohndorf2010} considers a day--ahead bidding problem different from ours using an infinite horizon MDP; \cite{Conejo2002} solves a price--taker bidding problem using a deterministic look--ahead policy; \cite{Gross2000} formulate a constrained optimization problem for optimal bidding in a competitive power pool; and \cite{David1993} develops both deterministic and stochastic models for bidding under the consideration of other market players. Lastly, \cite{Lohndorf2014} uses approximate dual dynamic programming (ADDP) to solve a day--ahead bidding problem involving hydro storage. Besides the major algorithmic differences from our paper, \cite{Lohndorf2014} also works in a day--ahead setting with individual bids for each hourly subinterval, while we work in an hourly setting with bids that must be simultaneously active for every 5 minute subinterval. Furthermore, in order to have complete information to make the optimal decision and to implement the transition dynamics, the previous bid (placed in the last time interval) is a part of our state variable, which is not the case for \cite{Lohndorf2014}. For more details, the literature survey by \cite{Wen2000} provides an excellent overview to strategic bidding.




In the case of real--world problems with large state spaces, backward dynamic programming is typically not a viable solution strategy, so we often use approximate dynamic programming (ADP) techniques. In this paper, we consider a variant of the approximate value iteration (AVI) algorithm (see both \cite{Bertsekas1996} and \cite{Powell2011}) that exploits the monotonicity in certain dimensions of the optimal value function (also known as the \emph{cost--to--go} function) in order to quickly approximate the shape of the value function. The algorithm, called \emph{Monotone--ADP}, is analyzed in \cite{Jiang2013} and was used previously as a heuristic in \cite{Papadaki2003}.

Like monotonicity, convexity/concavity also often arise in applications, and similar algorithms to Monotone--ADP that exploit these structural properties have been studied in \cite{Godfrey2001}, \cite{Topaloglu2003}, \cite{Powell2004}, and \cite{Nascimento2009a}. In general, the above studies on monotonicity and convexity have shown that it is advantageous to use the structural properties of value functions in ADP algorithms.

\section{Mathematical Formulation}
\label{sec:mathformulation}
We can formulate the problem mathematically as follows. Let $M$ be the number of settlements made per hour and let $\Delta t = 1/M$ be the time increment between settlements (in hours). For example, in the New York Independent System Operator (NYISO), settlements occur every 5 minutes, so we choose $M=12$. Although settlements are made intra--hour, bidding decisions are always made on the hour, for an hour in the future. Thus, the operator places bids at 1pm to operate the battery between 2 and 3pm, with settlements made in 5 minute intervals within the hour. For time indexing, we use $t$ (measured in hours); bidding decisions are made when $t \in \mathbb N$ and settlements occur when $t \in \mathcal T = \{k\cdot \Delta t: k \in \mathbb N\}$.

Let the price $P_t$ for $t \in \mathcal T$ be a discrete--time, nonnegative, stochastic process. Due to the fact that bidding decisions and settlements occur on two different schedules (every hour versus every $\Delta t$), we use the following notation. For $t \in \mathbb N$, let $P_{(t,t+1]}$ be an $M$--dimensional vector that represents the spot prices that occurred within the hour from $t$ to $t+1$:
\begin{equation}
P_{(t,t+1]} = (P_{t+\Delta t}, P_{t+2\cdot\Delta t}, \ldots, P_{t+(M-1)\cdot \Delta t},P_{t+1}).
\end{equation}
Hence, $P_{(t,t+1]}$ does not become fully known until time $t+1$. Next, let our set of bidding decisions be a finite set $\mathcal B$ such that
\begin{equation}
\mathcal{B} \subseteq \{(b^-,b^+):0 \le b^- \le b^+ \le b_\text{max}\},
\end{equation}
with $b_\text{max} \in \mathbb R_+$. Let $b_t = (b_t^-, b_t^+) \in \mathcal{B}$ be the bidding decision made at $t$ used for the interval $(t+1,t+2]$. All sell bids, $b_t^+$, (or ``positive'' bids because we are transferring at a rate of +1 MW) less than the spot price are picked up for dispatch (releasing energy into the grid). All buy bids, $b_t^-$, (or ``negative'' bids because we are transferring at a rate of $-1$ MW) greater than the spot price are picked up for charge. If the spot price falls in between the two bids, we are out of the market and the battery stays in an idle state. When we are obligated to sell to the market but are unable to deliver, we are penalized $K \cdot P_t$, where $K \ge 0$. The restriction of $b_t^- \le b_t^+$ guarantees that we are never obligated to buy and sell simultaneously.

We remark that in the actual bidding process, the buy bid is a negative number and the criteria for clearing the market is that the bid is less than the negative of the spot price. Due to our bids being for only two quantities ($\pm1$ MW), the above reformulation of the bidding process is cleaner and more intuitive.

Let $R_t \in \mathcal{R} = \{0,1,2,\ldots,\Rmax\}$ be the energy stored in the battery. For simplicity, assume that $\Rmax$ is adjusted so that a unit of resource represents $1/M$ MWh of energy. Thus, it is clear that at each settlement within the hour, the change in resource is either $+1$, $-1$, or $0$. We also define a deterministic function that maps a vector of intra--hour prices $P \in \mathbb{R}_+^M$ and a bid $b=(b^-,b^+) \in \mathcal{B}$ to a vector of outcomes (charge $=-1$, discharge $=+1$, or idle $=0$). Define $q: \mathbb{R}^M \times \mathcal B \rightarrow \{-1,0,1\}^M$ such that the $m$--th component of $q(P,b)$ is 
\begin{equation}
q_m(P,b)=\mathbf{1}_{\{b^+ < \,e_m^\intercal P\}}-\mathbf{1}_{\{b^- > \, e_m^\intercal P\}},
\label{chargefunc}
\end{equation}
where $e_m$ is a vector of zeros with a one at the $m$--th row (and thus, picks out the $m$--th component of the price vector $P$). Note that $q$ is not dependent on time, but in the context of our hour--ahead bidding problem, we use it in the form of $q(P_{(t-1,t]},b_{t-2})$, which is deterministic at time $t$. Figure \ref{operation} illustrates the intra--hour behavior.

\begin{figure}[h]
	\begin{center}
	\includegraphics[scale=.7]{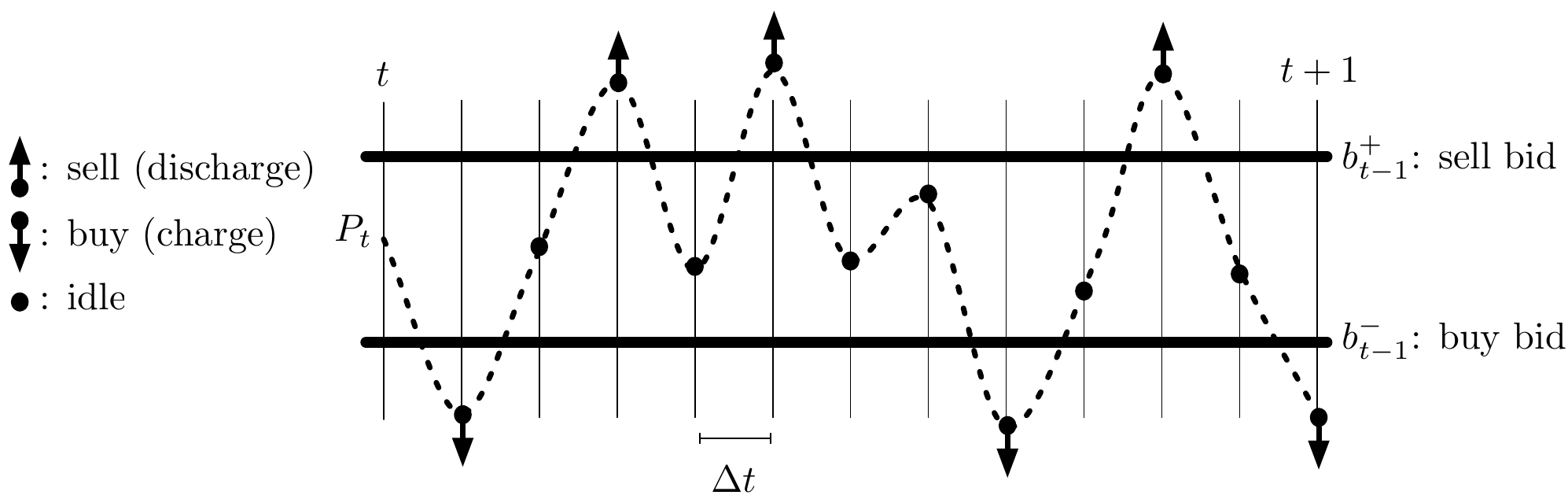}\\
	\end{center}
	\caption{Illustration of the Intra--hour Bidding Behavior}
	\label{operation}
\end{figure}

To define the hourly transition function between $R_t$ and $R_{t+1}$, we model each of the individual settlements within the hour and then combine them recursively (since from $t$ to $t+1$, we settle $M$ times). Let $q_s \in  \{-1,0,1\}^M$ be a vector of settlement outcomes and suppose $g^R_m(R_t,q_s)$ represents the amount of resource after the $m$--th settlement. Thus, we have
\begin{equation}
\begin{aligned}
g^R_0(R_t,q_s) &= R_t,\\
g^R_{m+1}(R_t,q_s) &= \bigl[\min \{g^R_m(R_t,q_s)-e_m^\intercal q_s, \Rmax\}\bigr]^+,
\end{aligned}
\label{gmtrans}
\end{equation}
for $1 \le m \le M$. The intra--hour resource levels are
\begin{equation*}
R_{t+m\Delta t} = g^R_m(R_t,q_s).
\end{equation*}
Finally, let $g^R$ be the hourly transition function, which is defined as a composition of the functions $g^R_M$ and $q$ in the following way:
\begin{equation}
R_{t+1} = g^R(R_t,P_{(t,t+1]},b_{t-1}) = g^R_{M}\bigl(R_t,q(P_{(t,t+1]},b_{t-1})\bigr).
\label{res_trans}
\end{equation}
The need for an hourly transition function from $R_t$ directly to $R_{t+1}$ (rather than simply defining the sub--transitions between the intra--hour settlements) is due to the hourly decision epoch of the problem.

\begin{restatable}{prop}{gmono}
For an initial resource level $r \in \mathcal R$, a vector of intra--hour prices $P \in \mathbb R^M$, a bid $b = (b^-,b^+) \in \mathcal B$, and a subinterval $m$, the resource transition function
$g^R_{m}\bigl(r,q(P,b)\bigr)$ is nondecreasing in $r$, $b^-$, and $b^+$.
\label{gmono}
\end{restatable}

We now consider another dimension to our problem by allowing a limit to be imposed on the number of charge--discharge cycles used by the battery, for the sake of increasing the lifetime of the battery. Battery cycle--life (the approximate number of cycles before capacity diminishes to around 80\%), a key issue when considering economic feasibility, varies between the different types of battery technologies and the operating conditions, but are typically in the range of 1000 (e.g., lead--acid) to 5000 (e.g., vanadium redox); for an extensive review, see \cite{Yang2011}. In our correspondence with industry colleagues, we found that a common (though possibly somewhat conservative) estimate of battery usage is 300 cycles/year, meaning that most devices can last at least 3 to 4 years. However, the model developed in this paper is for hourly decision making and it would be impractical to solve the model for time--horizons of several years. Note that different batteries technologies degrade in different ways, but in general, degradation occurs slowly (nearly linearly with charge--discharge cycles) at first, but after a point, efficiency drops much more rapidly.

Over a short horizon (on the order of days), the effects of battery degradation is negligible, but we propose the following way for one to impose a sort of artificial limit to the number of trades (charge--discharge cycles) performed. Let $L_t \in \mathcal L = \{0,1,2,\ldots,\Lmax\}$ be decremented on every discharge of the battery (starting with $L_0 = \Lmax$) and suppose that when selling to the market at a settlement time $t'$ in $(t,t+1]$, the revenue is discounted by a factor of $\beta(L_{t'})$ where $\beta:\mathcal L \rightarrow [0,1]$ is a nondecreasing function. Depending on the battery technology, preferences of the operator, and the time--horizon of the model, the choice of $\beta$ may vary greatly; the list below offers a few examples:
\begin{enumerate}
\item Constant: $\beta(l) = c \in [0,1]$ for all $l \in \mathcal L$,
\item Step: $\beta(0) = 0$ and $\beta(l) = 1$ for $l \in \mathcal L \setminus \{0\}$,
\item Linear: $\beta(l) = l/\Lmax$ for all $l \in \mathcal L$,
\item Power: $\beta(l) = (l/\Lmax)^\frac{1}n$ for some $n >1$ and all $l \in \mathcal L$,
\end{enumerate}
where (4) seeks to very roughly mimic the efficiency degradation of a real battery. We assume that the physical characteristics of the battery are summarized through $\beta$ and the dynamics of $L_t$, which we now describe.

Similar to the specification of $q$ in (\ref{chargefunc}), we define a function $d: \mathbb R^M \times \mathcal B \rightarrow \{0,1\}^M$ such that the $m$--th component of $d(P,b)$ is
\begin{equation*}
d_m(P,b)=\mathbf{1}_{\{b^+ < \,e_m^\intercal P\}},
\end{equation*}
which indicates the settlements for which a discharge occurred. Like before, we define the transition function from $L_t$ to $L_{t+1}$ using a sequence of sub--transitions. Let $d_s \in \{0,1\}^M$ be a vector of settlement outcomes (in this case, whether a discharge happened or not) and
\begin{equation}
\begin{aligned}
g^L_0(L_t,d_s) &= L_t,\\
g^L_{m+1}(L_t,d_s) &= \bigl[g^L_m(L_t,d_s)-e_m^\intercal d_s\bigr]^+,
\end{aligned}
\label{glmtrans}
\end{equation}
for $1 \le m \le M$. The intra--hour values are
\begin{equation*}
L_{t+m\Delta t} = g^L_m(L_t,d_s),
\end{equation*}
and the hourly transition function $g^L$ is defined
\begin{equation}
L_{t+1} = g^L(L_t,P_{(t,t+1]},b_{t-1}) = g^L_{M}\bigl(L_t,d(P_{(t,t+1]},b_{t-1})\bigr).
\label{resl_trans}
\end{equation}

\begin{restatable}{prop}{glmono}
For an initial $l \in \mathcal L$, a vector of intra--hour prices $P \in \mathbb R^M$, a bid $b = (b^-,b^+) \in \mathcal B$, and a subinterval $m$, the transition function $g^L_{m}\bigl(l,d(P,b)\bigr)$
is nondecreasing in $l$, $b^-$, and $b^+$.
\label{glmono}
\end{restatable}

At time $t$, we can determine the revenue from the previous hour $(t-1,t]$, which depends on the initial resource $R_{t-1}$, the remaining lifetime $L_{t-1}$, the intra--hour prices $P_{(t-1,t]}$, and the bid placed in the previous hour, $b_{t-2}$. The revenue made at the $m$--th settlement depends on four terms, the price $P_{t+m\Delta t}$, the settlement outcome $q_m(P,b)$ (which establishes the direction of energy flow), a \emph{discount factor} $\gamma_m$ (due to $L_t$), and the \emph{undersupply penalty} $U_m$. Let $r \in \mathcal R$, $l \in \mathcal L$, $P \in \mathbb R^M$, and $b \in \mathcal B$. Since we discount only when selling to the market, let
\begin{equation}
\gamma_m(l,P,b) = \beta(l) \cdot \mathbf{1}_{\{q_m(P,b) = 1\}} + \mathbf{1}_{\{q_m(P,b) \ne 1\}}.
\label{eq:gamma}
\end{equation}
The undersupply penalty takes values of either 1 (no penalty) or $-K$ (penalty):
\begin{equation}
U_m(r,P,b) = \Bigl(1-(K+1) \cdot \textbf{1}_{\{r = 0\}} \cdot \textbf{1}_{\{q_m(P,b) = 1\}}\Bigr).
\label{eq:D}
\end{equation}
This penalization scheme reflects reality: the NYISO penalizes using a price--proportional penalty of $K=1$ (in addition to lost revenue), the reason being to uphold the market balance. When a market participant reneges on a promise to deliver energy to the market, it must pay the penalty of the quantity times the market price to correct the imbalance; this is equivalent to \emph{purchasing that energy from another generator} at the market price and delivering to the market.

Hence, we can write the following sum (over the settlements) to arrive at the hourly revenue, denoted by the function $C$:
\begin{equation}
\begin{aligned}
{C}&\bigl(R_{t-1},L_{t-1},P_{(t-1,t]},b_{t-2}\bigr) \\
&= \sum_{m=1}^M \gamma_m \bigl(L_{t-1+m\Delta t},P_{(t-1,t]},b_{t-2} \bigr) \cdot P_{t+m\Delta t} \cdot
q_m(P_{(t-1,t]},b_{t-2}) \cdot U_m \bigl(R_{t-1+m\Delta t},P_{(t-1,t]},b_{t-2}\bigr).
\end{aligned}
\label{Cdef}
\end{equation}
Note that $C$ is not time--dependent. The timeline of events and notation we use is summarized in Figure \ref{fig:bidding_notation}. The top half of Figure \ref{fig:bidding_notation} shows the contrast between when bids are placed and when bids are active: $b_t$ and $b_{t+1}$ are placed at times $t$ and $t+1$ (arrows pointing up), while $b_{t-1}$ is active for the interval $(t,t+1]$ and $b_{t}$ is active for the interval $(t+1,t+2]$. It also shows that the revenue function $C(R_t,L_t,P_{(t,t+1]},b_{t-1})$ refers to the interval $(t,t+1]$. The bottom half of Figure \ref{fig:bidding_notation} shows an example of the bidding outcomes, i.e., the output of $q(P_{(t,t+1]},b_{t-1})$. Finally, we emphasize that $M$ settlements (and thus, transitions) occur between consecutive values of $R_t$ and $L_t$ due to the discrepancy between the bidding timeline (hourly) and the settlement timeline (every five minutes).
\begin{figure}[h]
	\begin{center}
	\includegraphics[scale=.65]{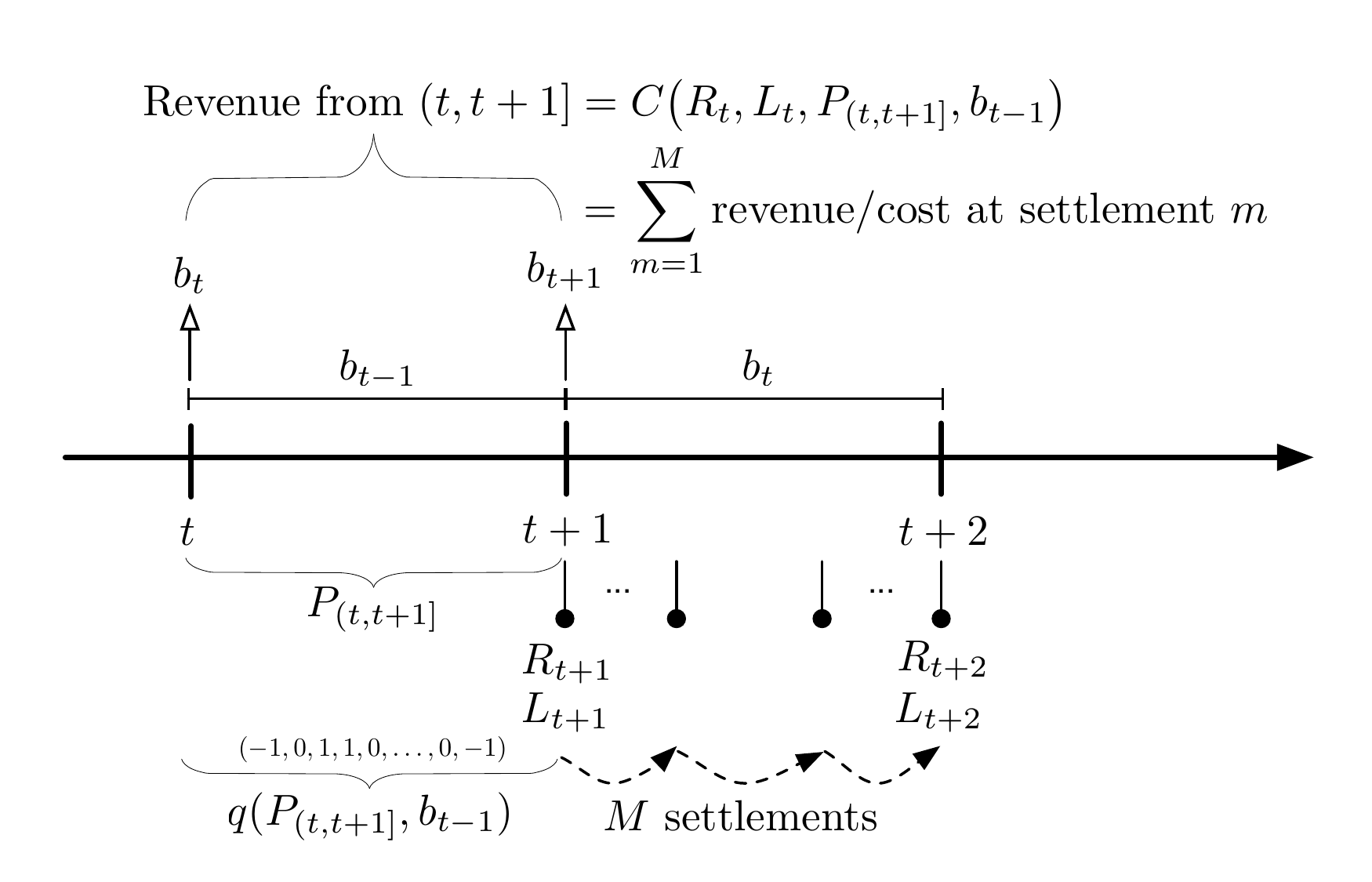}\\
	\end{center}
	\caption{Illustration of the Bidding Process}
	\label{fig:bidding_notation}
\end{figure}

\subsection{Markov Decision Process}The problem of optimizing revenue over a time horizon is a sequential decision problem that we can formulate as a Markov Decision Process (MDP). First, suppose the set of state variables associated with the price process $P_t$ is denoted $P^S_t \in \mathcal P$, where $\mathcal P$ is the space of price model state variables. The MDP can be characterized by the following components:
\begin{itemize}
\item[--] The \emph{state variable} for the overall problem is $S_t = (R_t, L_t, b_{t-1}^-, b_{t-1}^+, P^S_t) \in \mathcal S$ where $\mathcal S$ is the state space. The previous bid $b_{t-1}$ is included because it is the bid that becomes valid at time $t$ for the interval $(t,t+1]$ and is necessary for computing the resource transition function.
\item[--] The \emph{decision} is the hour--ahead bid $b_t = (b_t^-,b_t^+) \in \mathcal B$ that is active for the interval $(t+1,t+2]$.
\item[--] The \emph{exogenous information} in this problem is the price process $P_t$.
\item[--] The \emph{state transition function} or \emph{system model} $S^M$ is given by 
\begin{align}
S_{t+1} &= S^M(S_t,b_t,P_{(t,t+1]})\nonumber \\
&= \Bigl(g^R(R_t,P_{(t,t+1]},b_{t-1}), g^L(L_t,P_{(t,t+1]},b_{t-1}), b_t, P^S_{t+1} \Bigr).
\label{eq:statetrans}
\end{align}
\item[--] The \emph{contribution function} in this model represents the expected value of the revenue in the interval from $t+1$ to $t+2$ using bid $b_t$ given the current state $S_t$. Define:
\begin{equation}
C_{t,t+2}(S_t,b_t) = \mathbf{E}\Bigl[C(R_{t+1},L_{t+1},P_{(t+1,t+2]},b_t)\,|\,S_t\Bigr].
\label{Cbardef}
\end{equation}
The double subscript of $t$ and $t+2$ signifies that the contribution is determined at $t$ (hence, variables indexed by $t' \le t$ are known) but represents the expectation of the revenue in the interval $(t+1,t+2]$. In practice, it is likely that we must redefine $C_{t,t+2}(S_t,b_t)$ as a sample expectation over the available training data (see Section \ref{sec:casestudy}) if 1) a stochastic model of the prices is unavailable or 2) the expectation is impossible to compute. Regardless of its form, we assume that $C_{t,t+2}(S_t,b_t)$ can be computed exactly at time $t$.

\item[--] Let $T-1$ be the last time for which a bid needs to be placed (hence, the trading horizon lasts until $T+1$ and the last value function we need to define is at $T$) and let $B_t^\pi:\mathcal S \rightarrow \mathcal B$ be the decision function for a policy $\pi$ from the class $\Pi$ of all admissible policies. The following is the \emph{objective function} for maximizing expected revenue:
\begin{equation*}
\max_{\pi \in \Pi} \; \textbf{E} \left [ \sum_{t=0}^{T-1} C_{t,t+2}(S_t, B_t^\pi(S_t)) + C_\textnormal{term}(S_T) \,\Bigr|\, S_0 \right ], 
\end{equation*}
where $C_\textnormal{term}(S_T)$ represents a terminal contribution that is nondecreasing in $R_T$, $L_T$, and $b_{T-1}$.
\end{itemize}
We can express the optimal policy in the form of a stochastic dynamic program using Bellman's optimality equation \citep{Bellman1957a}. The optimal value function $V^*$ is defined for each $t$ and each state $S_t$:
\begin{equation}
\begin{aligned}
&V^*_t(S_t) =\max_{b_t \in \mathcal B} \Bigl [C_{t,t+2}(S_t,b_t)+\mathbf{E}\bigl[V^*_{t+1}(S_{t+1})\,|\,S_t\bigl] \Bigr ] \text{ for } t=0,1,2,\ldots,T-1,\\
&V^*_{T}(S_{T}) = C_\textnormal{term}(S_T).
\end{aligned}
\label{bellmangen}
\end{equation}
Figure \ref{dp} illustrates the above notation. Notice that at any decision epoch $t$, both the contribution and value functions are looking one step ahead, i.e., from $t+1$ onwards, in the form of an expectation. Because of this, the \emph{revenue} from $t$ to $t+1$ become, in a sense, irrelevant. However, the link between the time periods comes from the dependence of $R_{t+1}$ and $L_{t+1}$ on $R_t$, $L_t$, and $b_{t-1}$ (and of course, the random prices). In other words, at time $t$, our bid has to be placed for $(t+1,t+2]$ with an uncertain amount of resource, $R_{t+1}$ in the battery. It is important to note that it is precisely because $C_{t,t+2}(S_t,b_t)$ \emph{does not} include the revenue made in $(t,t+1]$ that allows us to show the important structural property of monotonicity for $C_{t,t+2}$ in $b_{t-1}$ (see Proposition \ref{prop:Ctmono} in the next section).

\begin{figure}[h]
	\begin{center}
	\includegraphics[scale=.6]{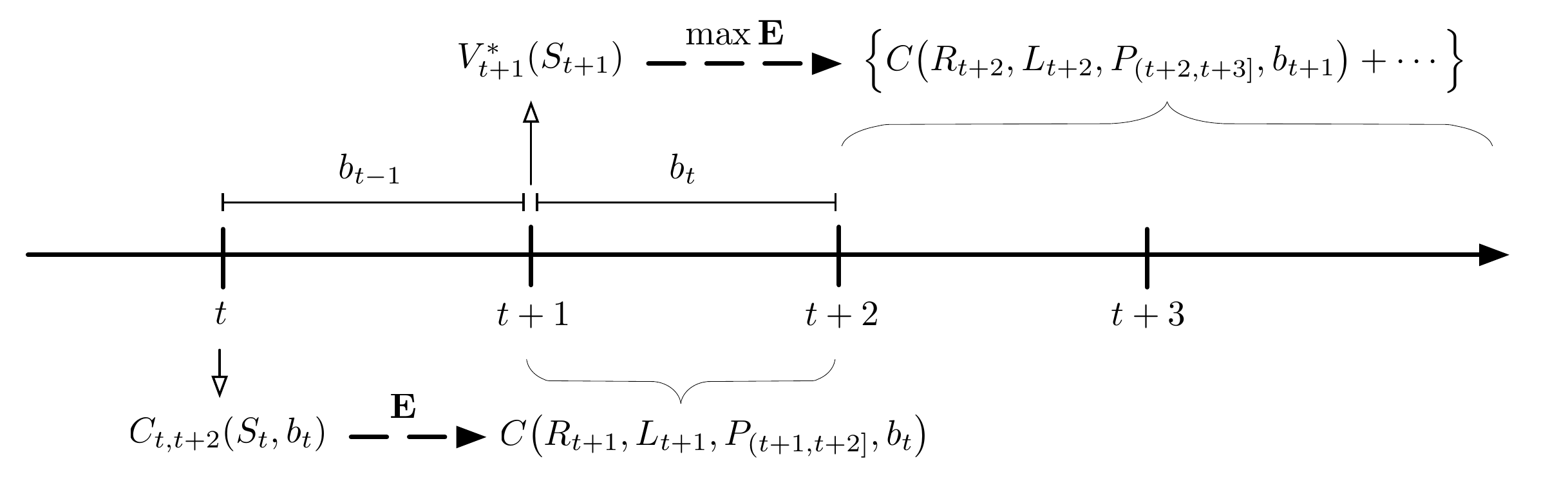}\\
	\end{center}
	\caption{Illustration of the Dynamic Programming Notation}
	\label{dp}
\end{figure}

We now provide some results regarding the structure of the contribution and value functions. The algorithm (Monotone--ADP--Bidding) that we implement to solve for the optimal value function is inspired by the following monotonicity properties.

\begin{restatable}{prop}{Ctmonoprop}
The contribution function ${C}_{t,t+2}(S_t,b_t)$, with $S_t = (R_t, L_t, b_{t-1},P_t^S)$ is nondecreasing in $R_t$, $L_t$, $b_{t-1}^-$, and $b_{t-1}^+$.\label{prop:Ctmono}
\end{restatable}

\begin{restatable}{prop}{Vtmonoprop}
The optimal value function $V^*_t(S_t)$, with $S_t = (R_t, L_t, b_{t-1},P_t^S)$ is nondecreasing in $R_t$, $L_t$, $b_{t-1}^-$, and $b_{t-1}^+$.\label{prop:Vtmono}
\end{restatable}

\section{Algorithmic Technique}
The traditional way to solve for the optimal value function in (\ref{bellmangen}) is by backward dynamic programming. Due to the fact that this technique requires us to visit every state (which is computationally difficult), we propose the use of approximate dynamic programming. We first note that both methods require a finite state space. Since $\mathcal R$, $\mathcal L$ and $\mathcal B$ were assumed to be finite, we need to assume, in particular, that $\mathcal P$ is also finite or that it is properly discretized.

The idea behind our ADP algorithm, which we call Monotone--ADP--Bidding (see \cite{Jiang2013}) is to iteratively learn the approximations $\widebar{V}^n_t(S_t)$ (after $n$ iterations) of $V_t^*(S_t)$ that obey the structural property of monotonicity. The algorithm is a form of asynchronous (or approximate) value iteration (AVI), so for each time $t$ in iteration $n$, only one state $S_t^n$ is visited. In addition, at each step, we perform a monotonicity preservation step to ensure the approximation is in a sense, structurally similar to $V_t^*$. We show experimentally that failure to maintain monotonicity, despite the availability of convergence proofs, produces an algorithm that simply does not work in practice.

\subsection{Preliminaries}Let $\hat{v}_t^n(S_t^n)$ be an observation of the value of being in state $S_t^n$ at iteration $n$ and time $t$. Define the noise term
\begin{equation*}
w_t^n(S_t^n) = \hat{v}_t^n(S_t^n) - \max_{b_t \in \mathcal B} \Bigl[ C_{t,t+2}(S^n_t,b_t)+\mathbf{E}\bigl[\widebar{V}^{n-1}_{t+1}(S_{t+1})\,|\,S^n_t\bigr] \Bigr],
\end{equation*}
to be the difference between the observation and the optimal value using the iteration $n-1$ \emph{approximation}. We remark, for the sake of clarity, that this is \emph{not} the noise representing the deviation from the true value, $V_t^*(S_t^n)$. Rather, $w_t^n(S_t^n)$ is the noise from an inability to exactly observe the optimal value of the maximization: $\max_{b_t \in \mathcal B} \bigl [C_{t,t+2}(S_t^n,b_t)+\mathbf{E}\bigl[\widebar{V}^{n-1}_{t+1}(S_{t+1})\,|\,S_t^n\bigr] \bigr ]$.

Thus, we can rearrange to arrive at:
\begin{equation*}
\hat{v}_t^n(S_t^n) =\max_{b_t \in \mathcal B} \Bigl [C_{t,t+2}(S_t^n,b_t)+\mathbf{E}\bigl[\widebar{V}^{n-1}_{t+1}(S_{t+1})\,|\,S_t^n\bigr] \Bigr ]+w_t^n(S_t^n).
\end{equation*}
Before we continue, let us define a partial order $\preccurlyeq$ on the state space $\mathcal S$ so that for $s = (r, l, b, p)$ and $s' = (r',l',b',p')$ where $r,r' \in \mathcal R$, $l,l' \in \mathcal L$, $b,b' \in \mathcal B$, and $p,p' \in \mathcal P$, we have that $s \preccurlyeq s'$ if and only if the following are satisfied:
\begin{equation*}
(r,l,b) \le (r', l', b') \textnormal{ and } p = p'.
\end{equation*}
The values of any two states that can be related by $\mless$ can be compared using Proposition \ref{prop:Vtmono}. The main idea of the algorithm is that every observation $\hat{v}_t^n(S_t^n)$ is smoothed with the previous estimate of the value of $S_t^n$ and the resulting smoothed estimate $z_t^n(S_t^n)$ can be used to generalize to the rest of the state space by means of a monotonicity preserving operator, $\Pi_M$. Let $s \in \mathcal S$ be an arbitrary state that has a current estimated value of $v$. After $z_t^n(S_t^n)$ is known, $\Pi_M$ adjusts the value of $s$ in the following way:
\begin{equation}
\Pi_M(S_t^n,z_t^n,s,v) = \left\{
	\begin{array}{ll}
		z_t^n & \mbox{if } s = S_t^n, \vspace{.6em} \\
		z_t^n \vee v  & \mbox{if } S_t^n \mless s ,\;s \ne S_t^n,
\vspace{.6em} \\
		z_t^n \wedge v & \mbox{if } s \mless S_t^n,\;s \ne S_t^n,
\vspace{.6em} \\
	v & \mbox{otherwise.}
	\end{array}
\right.
\label{genprojection}
\end{equation}

First, we note that if monotonicity is already satisfied, then nothing changes because in the second and third cases of (\ref{genprojection}), we get that $z_t^n \vee v = v$ and $z_t^n \wedge v = v$, respectively. If, however, monotonicity is violated, then the newly observed value $z_t^n$ prevails and replaces the previous value of $v$. Figure \ref{fig:projection} shows an example of this operation for the two bids $b_{t-1}^-$ and $b_{t-1}^+$. In the illustration, assume that the observations are made for fixed values of $R_t$ and $L_t$, but note that when we run the algorithm, this adjustment is made over all four dimensions. The figure should be interpreted as a three--dimensional plot of the value function, where all state variables besides $b_{t-1}$ are fixed. Each bid pair, $b_{t-1}=(b_{t-1}^-,b_{t-1}^+)$, is associated with a $z$--coordinate value represented by colors in gray scale (darker colors correspond to larger values). In the first and third plots, new observations arrive, and in the second and fourth plots, we see how the $\Pi_M$ operator uses monotonicity to generalize the observed values to the rest of the state space.

\begin{figure}[h]
	\begin{center}
	\includegraphics[scale=.4]{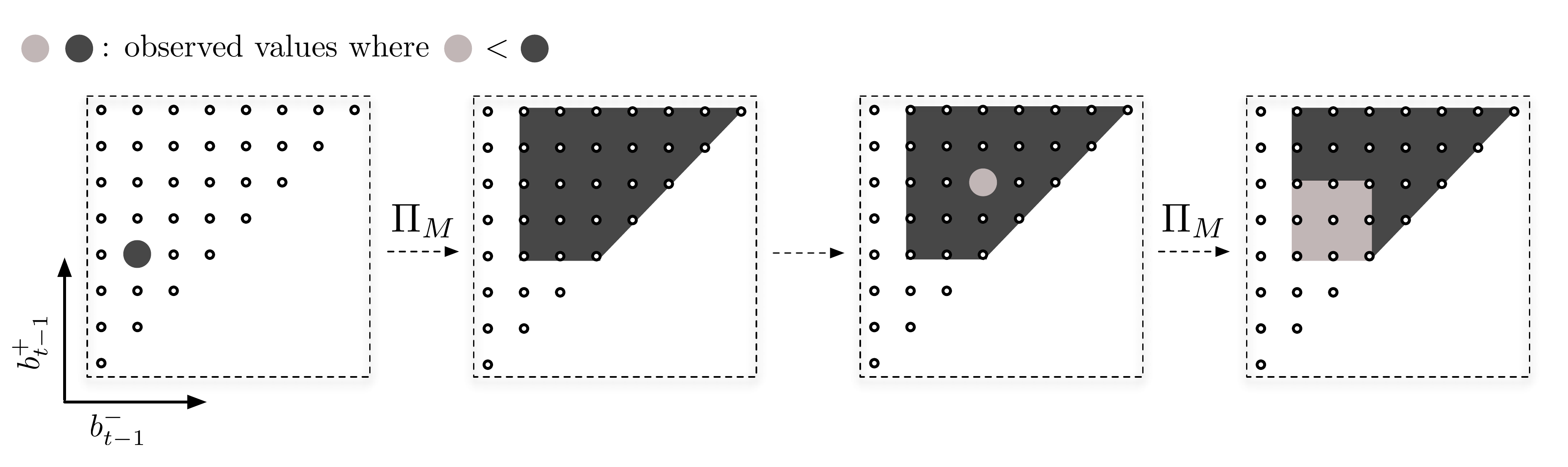}\\
	\end{center}
	\caption{Illustration of Monotonicity Preservation (Darker Colors = Larger Values)}
	\label{fig:projection}
\end{figure}
%
The stepsize sequence used for smoothing in new observations with the previous approximation is denoted $\alpha_t^n$, which can be thought of as a (possibly stochastic) sequence in $n$, for each $t$. Furthermore, states that are not visited do not get updated unless the update is made through the operator $\Pi_M$, so we also define:
\begin{equation*}
\alpha_t^n(s) = \alpha_t^{n-1} \, \textbf{1}_{\{s=S_t^n\}}.
\end{equation*}
For notational purposes, let us also define the history of the algorithm until iteration $n$ by the filtration
\begin{equation*}
\mathcal F^n = \sigma \bigl\{\bigl(S_{t}^m,\; w_{t}^m(S_{t}^m)\bigr)_{\; m\le n,\; t \le T}\bigr\}.
\end{equation*}
\subsection{Algorithm Description and Convergence}
The full description of the algorithm is given in Figure \ref{fig:algorithm}.
\begin{figure}[tb,h] 
\mbox{}\hrulefill\mbox{}
\begin{description}[leftmargin=5.2em,style=nextline]
    \item[Step 0a.]Initialize $\widebar{V}_t^0(s) = 0$ for each $t \le T-1$ and $s \in \mathcal S$.
    \vspace{.3em}
    \item[Step 0b.]Set $\widebar{V}_{T}^n(s) = 0$ for each $s \in \mathcal{S}$ and $n \le N$.
    \vspace{.3em}
    \item[Step 0c.]Set $n=1$.
    \vspace{.3em}    
    \item[Step 1.]Select an initial state $S_0^n$.
    \vspace{.3em}
    \item[Step 2.]For $t=0,1, \ldots, (T-1)$:
  	\vspace{.3em}    
    \begin{description}[leftmargin=5.2em,style=nextline]
    	\item[Step 2a.] Sample a noisy observation:\\
    	\vspace{.3em}
    	\quad $\displaystyle \hat{v}_t^n(S_t^n) =\max_{b_t \in \mathcal B} \Bigl [C_{t,t+2}(S_t^n,b_t)+\mathbf{E}\bigl[\widebar{V}^{n-1}_{t+1}(S_{t+1})\,|\,S_t^n\bigr] \Bigr ]+w_t^n(S_t^n)$.
    	\vspace{.3em}
    	\item[Step 2b.] Smooth in the new observation with previous value:\\
    	\vspace{.3em}
    	\quad $z_t^n(S_t^n) = \bigl(1-\alpha_t^{n}(S_t^n)\bigr) \, \widebar{V}_t^{n-1}(S_t^n)+\alpha_t^{n}(S_t^n) \, \hat{v}_t^n(S_t^n)$.
    	\vspace{.3em}
    	\item[Step 2c.] Perform monotonicity projection operator. For each $s \in \mathcal S$:\\
    	\vspace{.3em}
    	\quad $\widebar{V}_t^n(s) = \Pi_M(S_t^n,z_t^n,s,\widebar{V}_t^{n-1}(s))$.
    	\vspace{.3em}
    	\item[Step 2d.] Choose the next state $S_{t+1}^n$ given $\mathcal F^{n-1}$.
    \end{description}
    \vspace{.5em}
    \item[Step 3.] If $n < N$, increment $n$ and return \textbf{Step 1}.
\end{description}
 \mbox{}\hrulefill\mbox{}
\caption{Monotone--ADP--Bidding Algorithm for Training a Bidding Policy}
\label{fig:algorithm}
\end{figure}
Monotone--ADP--Bidding can be shown to converge; we reproduce the set of assumptions and the resulting theorem here.
\begin{assumption}
For all $s \in \mathcal S$ and $t \le T$,
\begin{equation*}
\sum_{n=1}^\infty \mathbf{P}\bigl(S_t^n = s\, |\, \mathcal F^{n-1}\bigr) = \infty \quad a.s.,
\end{equation*}
i.e, every state is visited infinitely often (see the Extended Borel--Cantelli Lemma in \cite{Breiman1992}).
\label{ass:one}
\end{assumption}

\begin{assumption}
The optimal value function $V_t^*(s)$ and the observations $\hat{v}_t^n(S_t^n)$ are bounded above and below by $\pm \Vmax$, where $\Vmax$ is a positive constant.
\label{ass:two}
\end{assumption}
\begin{assumption}
The noise sequence $w_t^n$ satisfies $\mathbf{E}\bigl[w_t^{n+1}(s)\,|\,\mathcal F^n\bigr] = 0$.
\label{ass:three}
\end{assumption}

\begin{assumption}
For each $t \le T$ and state $s$, suppose $\alpha_t^n \in [0,1]$ is $\mathcal F^n$--measurable and
\vspace{0.5em}
\begin{enumerate}[label=(\roman*),labelindent=1in]
\item $\displaystyle \sum_{n=0}^\infty \alpha_t^n(s) = \infty \quad a.s.$,
\vspace{0.5em}
\item $\displaystyle \sum_{n=0}^\infty \alpha_t^n(s)^2 < \infty \quad a.s.$\\
\end{enumerate}\label{ass:four}
\end{assumption}
\begin{thm}
Under Assumptions \ref{ass:one}--\ref{ass:four}, for each $t \le T$ and $s \in \mathcal S$, the estimates $\widebar{V}_t^n(s)$ produced by the Monotone--ADP--Bidding Algorithm of Figure \ref{fig:algorithm}, converge to the optimal value function $V_t^*(s)$ almost surely.
\label{maintheorem}
\end{thm}
\begin{proof}
The proof is based on the result for a generalized MDP with a monotone value function in \cite{Jiang2013}. 
\end{proof}

\subsection{Approximating the Expectation}
Our algorithm can be applied to any model of spot prices $P_t$, with the caveat that more complex models generally require a higher dimensional state space. These include diffusion models (i.e., \cite{Schwartz1997}, \cite{Cartea2005}, \cite{Coulon2012}), which often incorporate features such as Ornstein--Uhlenbeck processes, jump terms, and regime switching. Recently, there has also been interest in structural models of electricity prices, where the underlying supply, demand, and bid--stack behavior is taken into account; see \cite{Carmona2012} for a survey.

The fact that the state space becomes larger or higher dimensional is indeed a computational difficulty that requires the availability of more powerful computational resources, but the convergence of the algorithm is unaffected (as long as $P_t^S$ is properly discretized). On the other hand, any model without finite support (or finite, but with large cardinality) necessitates the approximation of an expectation using a sample mean in Step 2a of the Monotone--ADP--Bidding algorithm (see Figure \ref{fig:algorithm}). In other words, because the expectation
$\mathbf{E}\bigl(\widebar{V}^{n-1}_{t+1}(S_{t+1})\,|\,S_t^n\bigr)$ of Step 2a is, in general, impossible to compute, we must resort to letting $v_t^n(S_t^n)$ be the solution to the \emph{sample average approximation} (see \cite{Kleywegt2002}) problem:
\begin{equation}
\max_{b_t \in \mathcal B} \biggl [C_{t,t+2}(S_t^n,b_t)+ J^{-1}\sum_{j=1}^J\widebar{V}^{n-1}_{t+1}(S^j_{t+1}) \biggr ],
\label{eq:saa}
\end{equation}
where $S_{t+1}^j$ are samples drawn independently from the distribution $S_{t+1}\,|\,S_t^n$. Suppose we take the observation $\hat{v}_t^n(S_t^n)$ to be the value of (\ref{eq:saa}). By an interchange of the conditional expectation and the max operator, we see that:
\begin{align*}
\mathbf{E} \biggl [\max_{b_t \in \mathcal B} \Bigl [C_{t,t+2}(S_t^n,b_t) &+ J^{-1}\sum_{j=1}^J\widebar{V}^{n-1}_{t+1}(S^j_{t+1}) \Bigr] \, \Bigr| \, S^{n}_t\biggr] \ge\\
 &\max_{b_t \in \mathcal B} \Bigl [C_{t,t+2}(S_t^n,b_t)+ \mathbf{E}\bigl(\widebar{V}^{n-1}_{t+1}(S_{t+1}) \, \bigr|\, S_t^n \bigr) \Bigr] \quad a.s.,
\end{align*}
and thus, after conditioning on $\mathcal F^{n-1}$ on both sides, we see that $\mathbf{E}[w_t^n(S_t^n)|\mathcal{F}^{n-1}]$ is biased upward from zero, a contradiction of Assumption \ref{ass:three}. When $J$ is large, we can certainly solve the sample average approximation problem in Step 2a and apply the algorithm as is and expect an effective heuristic solution. Practically speaking, our informal tests (using $J=1000$ on a diffusion price model) showed no significant convergence issues. Even so, we cannot claim that such an approximation produces a theoretically sound and convergent algorithm due to the biased noise term. This calls for us to propose another version of Monotone--ADP, for which Assumption \ref{ass:three} can be easily satisfied, without restricting to price process models that facilitate an easily computable expectation of the downstream value. 

\subsection{Post--Decision, Distribution--Free Approach}
\label{sec:pdalg}
Using the idea of a post--decision state (see \cite{Powell2011}), we can make a small adjustment to the algorithm, so that Assumption \ref{ass:three} is satisfied. In the case of the hourly bidding problem, the post--decision state $S_t^b$ is the state--action pair $(S_t, b_t)$. Oftentimes, post--decision states help simplify the computational aspect of an MDP, but unfortunately, for this problem instance, the post--decision state space is higher dimensional than the pre--decision state space. Let $S_t^b = (S_t,b_t) \in \mathcal S^b$ and define the post--decision value function
\begin{equation*}
V_t^b(S_t^b)=V_t^b(S_t,b_t) = \mathbf{E}\bigl[V^*_{t+1}(S_{t+1})\,|\,S_t^b\bigr].
\end{equation*}
Notice that we can rewrite Bellman's optimality equation as:
\begin{equation}
V^b_{t-1}(S_{t-1}^b) = \mathbf{E} \Bigl[\max_{b_t \in \mathcal B} \bigl [C_{t,t+2}(S_t,b_t)+V_t^b(S_t^b) \bigr ] \,|\, S_{t-1}^b \Bigr ].
\label{bellmanpost}
\end{equation}
Instead of attempting to learn $V_t^*$, the idea now is to algorithmically learn the post--decision value function $V_t^b$ using the relation (\ref{bellmanpost}) and to implement the policy by solving
\begin{equation*}
b_t^*=\argmax_{b_t \in \mathcal B} \bigl [C_{t,t+2}(S_t,b_t)+V_t^b(S_t^b) \bigr ].
\end{equation*}
Not surprisingly, the post--decision value function $V_t^b$ also satisfies a monotonicity property, over six dimensions.
\begin{restatable}{prop}{Vtpostmonoprop}
The post--decision value function $V^b_t(S_t^b)$, with $S_t^b = (R_t, L_t, b_{t-1},b_t,P_t^S)$ is nondecreasing in $R_t$, $L_t$, $b_{t-1}^-$, $b_{t-1}^+$, $b_t^-$, and $b_t^+$.\label{prop:Vtpostmono}
\end{restatable}

Let $\widebar{V}_t^{b,\,n}$ be the iteration $n$ approximation of the post--decision value function, $S_t^{b,\,n}$ be the state visited by the algorithm in iteration $n$, $\hat{v}_{t}^{b,\,n}(S_{t}^{b,\,n})$ be an observation of $V_{t}^b(S_{t}^{b,\,n})$ using the iteration $n-1$ approximation, $w_{t}^{b,\,n}(S_{t}^{b,\,n})$ be the observation noise, $\mathcal F^{b,\,n}$ be a filtration defined analogously to $\mathcal F^n$, and $\Pi_M^b$ be the monotonicity preservation operator on $\mathcal S^b$ defined analogously to $\Pi_M$. More precisely,
\begin{equation}
\hat{v}_{t}^{b,\,n}(S_{t}^{b,\,n}) = 
    	    	\mathbf{E} \Bigl [ \max_{b_{t+1} \in \mathcal B} \bigl [C_{t+1,t+3}(S_{t+1},b_{t+1})+\widebar{V}_{t+1}^{b,\,n-1}(S_{t+1}^b) \bigr] \,\bigl|\, S_t^{b,n} \Bigr ] + w_{t}^{b,n}(S_t^{b,n}),
\end{equation}
\begin{equation}
\mathcal F^{b,n} = \sigma\bigl\{\bigl(S_{t}^{b,m},\; w_{t}^{b,m}(S_{t}^{b,m})\bigr)_{\; m\le n,\; t \le T}\bigr\},
\end{equation}
and
\begin{equation}
\Pi_M^b(S_t^{b,n},z_t^{b,n},s,v) = \left\{
	\begin{array}{ll}
		z_t^{b,n} & \mbox{if } s = S_t^{b,n}, \vspace{.6em} \\
		z_t^{b,n} \vee v  & \mbox{if } S_t^{b,n} \mless^b s ,\;s \ne S_t^{b,n},
\vspace{.6em} \\
		z_t^{b,n} \wedge v & \mbox{if } s \mless^b S_t^{b,n},\;s \ne S_t^{b,n},
\vspace{.6em} \\
	v & \mbox{otherwise,}
	\end{array}
\right.
\label{genprojection2}
\end{equation}
where $s = (r, l, b_1, b_2, p) \mless^b s' =  (r', l', b_1', b_2', p')$ with $r,r' \in \mathcal R$, $l,l' \in \mathcal L$, $b_1,b_1',b_2,b_2' \in \mathcal B$, and $p,p' \in \mathcal P$ if and only if 
\begin{equation*}
(r, l, b_1, b_2) \le (r', l', b_1', b_2') \textnormal{ and } p = p'.
\end{equation*}
The new algorithm for post--decision states is shown in Figure \ref{fig:algorithm2}, and a set of analogous assumptions are provided below. We remark that by definition, $V_{T-1}^b(S_{T-1}^b) = \mathbf{E}\bigl[ C_\textnormal{term}(S_{T})\, |\, S_{T-1}^b\bigr]$; thus, we only need to loop until $T-2$ in Step $2$.
\begin{figure}[tb,h] 
\mbox{}\hrulefill\mbox{}
\begin{description}[leftmargin=4.68em,style=nextline]
    \item[Step 0a.]Initialize $\widebar{V}_t^{b,\,0}(s) = 0$ for each $t \le T-1$ and $s \in \mathcal S^b$.
    \vspace{.3em}
    \item[Step 0b.]Set $\widebar{V}_{T}^{b,\, n}(s) = 0$ for each $s \in \mathcal{S}^b$ and $n \le N$.
    \vspace{.3em}
    \item[Step 0c.]Set $n=1$.
    \vspace{.3em}    
    \item[Step 1.]Select an initial state $S_{0}^{b,\,n}=(S_{0}^n,b_{0}^n)$.
    \vspace{.3em}
    \item[Step 2.]For $t=0, \ldots, (T-2)$:
  	\vspace{.3em}    
    \begin{description}[leftmargin=4.68em,style=nextline]
    	\item[Step 2a.] Sample a noisy observation:\\
    	\vspace{.3em}
    	$\displaystyle \hat{v}_{t}^{b,n}(S_{t}^{b,n}) =
    	    	\mathbf{E} \Bigl [ \max_{b_{t+1} \in \mathcal B} \bigl [C_{t+1,t+3}(S_{t+1},b_{t+1})+\widebar{V}_{t+1}^{b,\,n-1}(S_{t+1}^b) \bigr]  \,\bigl|\, S_{t}^{b,n}  \Bigr] + w_{t}^{b,n}(S_t^{b,n})$.
    	\item[Step 2b.] Smooth in the new observation with previous value:\\
    	\vspace{.3em}
    	\enspace $z_{t}^{b,\,n}(S_{t}^{b,\,n}) = \bigl(1-\alpha_{t}^{n}(S_{t}^{b,\,n})\bigr) \, \widebar{V}_{t}^{b,\,n-1}(S_{t}^{b,\,n})+\alpha_{t}^{n}(S_{t}^{b,\,n}) \, \hat{v}_{t}^{b,\,n}(S_{t}^{b,\,n})$.
    	\vspace{.3em}
    	\item[Step 2c.] Perform monotonicity preservation operator. For each $s \in \mathcal S^b$:\\
    	\vspace{.3em}
    	\enspace $\widebar{V}_{t}^{b,\,n}(s) = \Pi^b_M\bigl(S_{t}^{b,\,n},z_{t}^{b,\,n},s,\widebar{V}_{t}^{b,\,n-1}(s)\bigr)$.
    	\vspace{.3em}
    	\item[Step 2d.] Choose the next state $S_{t+1}^{b,\,n}$ given $\mathcal F^{b,\,n-1}$.
    \end{description}
    \vspace{.3em}
    \item[Step 3.] If $n < N$, increment $n$ and return \textbf{Step 1}.
\end{description}
 \mbox{}\hrulefill\mbox{}
\caption{Monotone--ADP--Bidding Algorithm using Post--Decision States}
\label{fig:algorithm2}
\end{figure}

\begin{assumption}
For all $s \in \mathcal S^b$ and $t \le T$,
\begin{equation*}
\sum_{n=1}^\infty \mathbf{P}\bigl(S_t^{b,n} = s \,|\, \mathcal F^{b,n-1}\bigr) = \infty \quad a.s.
\end{equation*}
\label{ass:five}
\end{assumption}
\begin{assumption}
The optimal post--decision value function $V_t^b(s)$ and the observations $\hat{v}_t^{b,n}(S_t^{b,n})$ are bounded above and below, by $\pm \Vmax$.
\label{ass:six}
\end{assumption}
\begin{assumption}
The noise sequence $w_t^{b,n}$ satisfies $\mathbf{E}\bigl[w_t^{b,n+1}(s)\,|\,\mathcal F^{b,n}\bigr] = 0.$
\label{ass:seven}
\end{assumption}

The advantage to applying this revised algorithm is that even when we cannot compute the expectation in Step 2a and must rely on sample paths, we can still easily satisfy Assumption \ref{ass:seven} (the unbiased noise assumption), unlike in the pre--decision case. To do so, we simply use:
\begin{equation*}
    	\hat{v}_{t}^{b,\,n}(S_{t}^{b,\,n}) = \max_{b_{t+1} \in \mathcal B} \Bigl [C_{t+1,t+3}(S^n_{t+1},b_{t+1})+\widebar{V}_{t+1}^{b,\,n-1}(S_{t+1}^n,b_{t+1}) \Bigr ],
\end{equation*}
where we transition from $S_{t}^{b,n}$ to $S_{t+1}^n$ using a \emph{single sample outcome} of prices $P_{(t,t+1]}$. Hence, the noise term $w_t^{b,n}(S_t^{b,n})$ is trivially unbiased.

Besides being able to work with more complex price models, the revised algorithm gives us another important advantage, especially for implementation in practice/industry. As long as historical data is available, a model of the real--time prices is \emph{not required} to train the algorithm. We propose an alternative idea: instead of fitting a stochastic model to historical data and then sampling $P_{(t,t+1]}$ from the model, we can simply take a price path directly from historical data. Since no specific knowledge regarding the distribution of prices is needed (besides the boundedness assumption needed for the convergence of the ADP algorithm), as previously mentioned, we refer to this as a \emph{distribution--free} approach, and the technique is employed in Section \ref{sec:casestudy}. We now state the convergence theorem for the post--decision state version of Monotone--ADP--Bidding. Because the convergence theory for the post--decision state version is not discussed in detail in \cite{Jiang2013}, we provide a sketch of the proof here. First, we define the following post--decision Bellman operator that acts on a vector of values $V \in \mathbb R^{T\cdot |\mathcal S^b|}$ (any $V$, not necessarily corresponding to the optimal value function), for $S_t^b \in \mathcal S^b$ and $t \le T$:
\begin{equation}
\bigl(HV\bigr)_t(S^b_t) = \left\{
	\begin{array}{ll}
		 \mathbf{E} \Bigl[\max_{b_{t+1} \in \mathcal B} \bigl [C_{t+1,t+3}(S_{t+1},b_{t+1})+V_{t+1}(S_{t+1}^b) \bigr ] \,|\, S_{t}^b \Bigr ]
		 & \mbox{for } t=0,1,2,\ldots,T-2, \vspace{.6em} \\
		 \mathbf{E}[ C_\textnormal{term}(S_{t+1}) \,|\, S_t^b] & \mbox{for } t=T-1. 
	\end{array}
\right.
\label{Hdef}
\end{equation}
Step 2a of the algorithm (Figure \ref{fig:algorithm2}) can thus be rewritten as:
\[
\hat{v}_{t}^{b,\,n}(S_{t}^{b,\,n}) = \bigl(H\widebar{V}^{b,n-1}\bigr)_t(S^{b,n}_t) + w_t^{b,n}(S_t^{b,n}).
\]

\begin{thm}
Under Assumptions \ref{ass:four}--\ref{ass:seven}, for each $t \le T$ and $s \in \mathcal S^b$, the estimates $\widebar{V}_t^{b,n}(s)$ produced by the post--decision version of Monotone--ADP--Bidding Algorithm of Figure \ref{fig:algorithm2}, converge to the optimal post--decision value function $V_t^b(s)$ almost surely.
\label{postdecisiontheorem}
\end{thm}
Before discussing the proof, we state two necessary lemmas (proofs in Appendix \ref{sec:appendix}). The idea of the first lemma is attributed to \cite{Tsitsiklis1994a}.

\begin{restatable}{lem}{lemone}
Define deterministic bounding sequences $L_t^k$ and $U_t^k$ in the following way. Let $U^0 = V^*+\Vmax\cdot e$ and $L ^0=V^*-\Vmax\cdot e$, where $e$ is a vector of ones. In addition, $U^{k+1}=(U^k+HU^k)/2$ and $L^{k+1} = (L^k+HL^k)/2$. Then, for each $s \in \mathcal S^b$ and $t \le T-1$,  
\begin{equation*}
\begin{aligned}
L_t^k(s) &\rightarrow V_t^b(s),\\
U_t^k(s) &\rightarrow V_t^b(s),
\end{aligned}
\label{eq:uvconv}
\end{equation*}
where the limit is in $k$.
\label{lem:one}
\end{restatable}

\begin{restatable}{lem}{lemtwo}
$U^k$ and $L^k$ both satisfy the monotonicity property: for each $t$, $k$, and $s_1, s_2 \in \mathcal S^b$ such that $s_1 \mless^b s_2$,
\begin{equation}
\begin{aligned}
U_t^k(s_1) &\le U_t^k(s_2),\\
L_t^k(s_1) &\le L_t^k(s_2).
\end{aligned}
\label{eq:uvmono}
\end{equation}
\label{lem:two}
\end{restatable}

\begin{proof}[Sketch of Proof of Theorem 2] With Lemmas \ref{lem:one} and \ref{lem:two}, we can proceed to show convergence of the post--decision state version of Monotone--ADP using the general steps to prove convergence of Monotone--ADP for \emph{pre--decision} states taken in \cite{Jiang2013}.
The steps are as follows:
\begin{enumerate}
\item Given a fixed $k$ and a state $s \in \mathcal S^b$ such that $s$ is increased \emph{finitely often} by the monotonicity preservation operator $\Pi_M^b$, then we can show that for any sufficiently large $n$,
\begin{equation}
L_t^k(s) \le \widebar{V}_t^{b,n}(s) \le U_t^k(s).
\label{eq:squeeze}
\end{equation} There exists at least one such state, i.e., the minimal state $(0,0,(\bmin,\bmin),P_t^S)$. Repeat the argument for states that are decreased finitely often by $\Pi_M^b$.
\item Next, we must show that states $s$ that are affected by $\Pi_M^b$ \emph{infinitely often} also satisfy (\ref{eq:squeeze}). This leverages the fact that the result has already been proven for states that are affected finitely often. The idea is that if all states immediately less than $s$ (i.e., $x$ is immediately less than $y$ if $x \mless^b y$ and there does not exist $z$ such that $x \mless^b z \mless^b y$) satisfy (\ref{eq:squeeze}), then $s$ satisfies (\ref{eq:squeeze}) as well. Lemma \ref{lem:two} and an induction argument are used in this part of the proof.
\item Finally, combining Lemma \ref{lem:one} along with the fact that all post--decision states $s \in \mathcal S^b$ satisfy (\ref{eq:squeeze}), it is easy to see that from a type of squeeze argument,
\[
\widebar{V}^{b,n}_t(s) \rightarrow V_t^b(s),
\]
for each $t$ and $s$, as desired.
\end{enumerate}
Note that both Steps (1) and (2) require Assumption \ref{ass:seven}, hence the focus that we have placed on it in this paper.
\end{proof}

\subsection{Stepsize Selection} The selection of the stepsize $\alpha_t^n$, also known as a \emph{learning rate}, can have a profound effect on the speed of convergence of an ADP algorithm. A common example of stepsize rule that satisfies Assumption \ref{ass:four} is simply:
\[
\alpha_t^n = \frac{1}{N(S_t^n,n)},
\]
where $N(S_t^n,n) = \sum_{m=1}^n \indicate{S_t^m=S_t^n}$ is the number of visits by the algorithm to the state $S_t^n$. The issue is that this method weighs all observations equally, even though we know that the error can be extremely large in early iterations of any ADP algorithm. See Chapter 11 of \cite{Powell2011} for an overview of the numerous available stepsize rules.

After some experimentation, we found that the bias--adjusted Kalman Filter (BAKF) developed in \cite{George2006}, performed better than simpler alternatives. The main idea behind BAKF is to choose $\alpha_t^n$ such that the mean squared error to the true value function is minimized; we omit the details and refer interested readers to the original paper.


\section{Benchmarking on Stylized Problems using Pre--Decision Monotone--ADP}
\label{sec:benchmarking}In this section, we present results of running Monotone--ADP--Bidding and traditional approximate value iteration on a tractable problem (i.e., the optimal solution is computable) in order to show the advantages of using $\Pi_M$. In this section, we consider both four and five dimensional versions of the sequential bidding problem. We first describe some simplifications to make benchmarking possible.

In order to benchmark the algorithm against a truly optimal solution, we make some simplifying assumptions (to be relaxed in the following section) so that backward dynamic programming can be used to compute an optimal solution. First, we suppose that $P_t$ has finite support and that $M=1$, so that the exact value of $\mathbf{E}\bigl[V_{t+1}(S_{t+1})\,|\,S_t\bigr]$ can be computed easily. When $M$ is larger, we can only compute an approximation to the expectation, due to the fact that an exponential in $M$ number of outcomes of the price process need to be considered for an exact result.

In addition, in the numerical work of this paper, we take the traditional approach and choose $C_\textnormal{term}(s)=0$; however, we remark that this may not always be the best choice in practice. See Section \ref{sec:additionalinsights} for further discussion on the issue of selecting a terminal contribution function.

To test the approximate policies, we compute a \emph{value of the policy} in the following way. For a particular set of value function approximations $\widebar{V}$, the set of decision functions can be written as
\begin{equation*}
\bar{B}_t(S_t) = \argmax_{b_t \in \mathcal B} \; \Bigl[C_{t,t+2}(S_t,b_t) + \mathbf{E}\bigl[\widebar{V}_{t+1}(S_{t+1})\,|\,S_t\bigr]\Bigr].
\end{equation*}
For a sample path $\omega \in \Omega$, let
\begin{equation*}
F\bigl(\widebar{V},\omega\bigr) = \sum_{t=0}^{T+1} {C}\bigl(R_{t+1}(\omega),L_{t+1}(\omega),P_{(t+1,t+2]}(\omega),\bar{B}_t(S_t)\bigr)
\end{equation*}
be a sample outcome of the revenue. We report the empirical value of the policy, which is the sample mean of $F\bigl(\widebar{V},\omega\bigr)$ over $1000$ sample paths $\omega$.

\subsection{Variation 1}
First, we consider a four dimensional variation of the bidding problem, where $S_t = (R_t, L_t, b_{t-1}^-, b_{t-1}^+)$. In particular, we assume that the price process has no state variables. Several versions of this problem are explored by altering the parameter values: a typical size for the batteries under consideration for the energy arbitrage application is $\Rmax = 6$ MWh, but we also allow values of $\Rmax = 12$ MWh and $\Rmax=18$ MWh for variety. The decision space is fixed in the following way: we set $\bmin = 15$ and $\bmax = 85$, and discretized linearly between $\bmin$ and $\bmax$ for a total of 30 possible values in each dimension of the bid. The price process $P_t$ has the form
\[
P_t = S(t) + \epsilon_t,
\]
where the sinusoidal (representing the hour--of--day effects on price) deterministic component is
\begin{equation*}
S(t) = 15 \, \sin(2\pi t/24)+50,
\end{equation*}
and $\epsilon_t \in \{0, \pm1, \pm2, \ldots, \pm20\}$, a sequence of mean zero i.i.d. random variables distributed according to the \emph{discrete pseudonormal distribution} with $\sigma_X^2 = 49$ (a discrete distribution where the probability masses are defined by the evaluating at the density function of $\mathcal N(0,\sigma_X^2)$ and then normalizing).
We consider both the cases where the battery age does and does not matter (by setting $\beta(l)=1$), in effect introducing an irrelevant state variable. When aging does matter, the aging function we use is $\beta(l) = (l/\Lmax)^{\frac16}$, which provides a roughly linear decline in efficiency from 100\% to around 70\%, followed by a much steeper decline. Lastly, in Problem 4, we considered a uniform distribution for the noise, while the remaining problems used pseudonormal noise. In line with the operation procedures of the NYISO, the undersupply penalty parameter $K$ is set to $1$ in our simulations --- this means that if one is unable to deliver energy to the market, then the penalty is precisely the current spot price (essentially, we are paying another generator to produce the energy instead). The different problems instances, labeled $A_1$--$F_1$, along with their state space cardinalities are summarized in Table \ref{table:problems1}. 


\begin{table}[h]
\centering
\tiny
\begin{tabular}{@{}ccccccccc@{}}\toprule
\textbf{Problem} & $T$ & $\Rmax$ & $\Lmax$ & $\beta(l)$ & \textbf{Distribution of} $\epsilon_t$ & \textbf{Cardinality of } $\mathcal S$\\
\midrule
 $A_1$ & $24$ & $6$ & $8$ & $1$ & Pseudonormal & 22{,}320 \vspace{0.2em} \\
$B_1$ & $24$ & $6$ & $8$ & $(l/8)^\frac{1}{6}$ & Pseudonormal & 22{,}320 \vspace{0.2em} \\
 $C_1$ & $36$ & $6$ & $8$ & $1$ & Pseudonormal & 22{,}320 \vspace{0.2em} \\
 $D_1$ & $24$ & $12$ & $12$ & $(l/12)^\frac{1}{6}$ & Uniform & 66{,}960 \vspace{0.2em} \\
 $E_1$ & $24$ & $12$ & $12$ & $(l/12)^\frac{1}{6}$ & Pseudonormal & 66{,}960 \vspace{0.2em} \\
 $F_1$ & $36$ & $18$ & $18$ & $(l/18)^\frac{1}{6}$ & Pseudonormal &	150{,}660 \vspace{0.2em} \\
\bottomrule
\end{tabular}
\vspace{1em}
\caption{Parameter Choices for Variation 1 Benchmark Problems}
\label{table:problems1}
\end{table}
%

\subsection{Numerical Results for Variation 1}
We first evaluate the effectiveness of Monotone--ADP--Bidding versus approximate value iteration, a traditional ADP algorithm (exactly the same as Monotone--ADP--Bidding with $\Pi_M$ removed); the results for Variation 1 are given in Table \ref{table:comparison2}.

Figure \ref{fig:compare} gives a quick visual comparison between the two types of approximate value functions, generated by approximate value iteration and Monotone--ADP--Bidding. We remark that after $N=1000$ iterations, the value function approximation in Figure \ref{subfig:adp} obtained by exploiting monotonicity has developed a discernible shape and structure, with a relatively wide range of values. The result in Figure \ref{subfig:avi}, on the other hand, is relatively unusable as a policy. 
\begin{figure}[ht]
        \centering
        \begin{subfigure}[b]{0.45\textwidth}
                \centering
                \includegraphics[width=\textwidth]{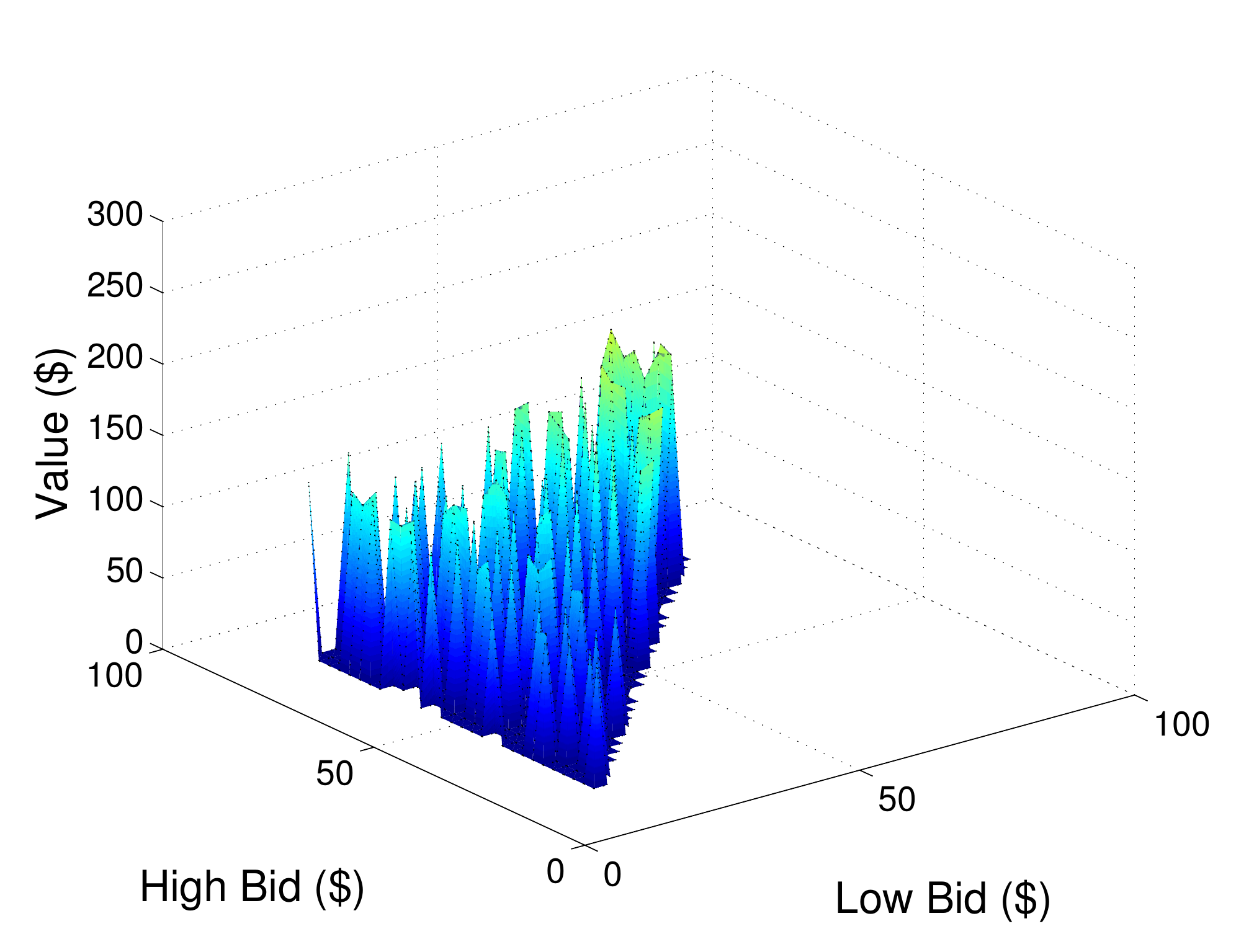}
                \caption{AVI, $N=1000$}\label{subfig:avi}
        \end{subfigure}%
        ~ 
        \begin{subfigure}[b]{0.45\textwidth}
                \centering
                \includegraphics[width=\textwidth]{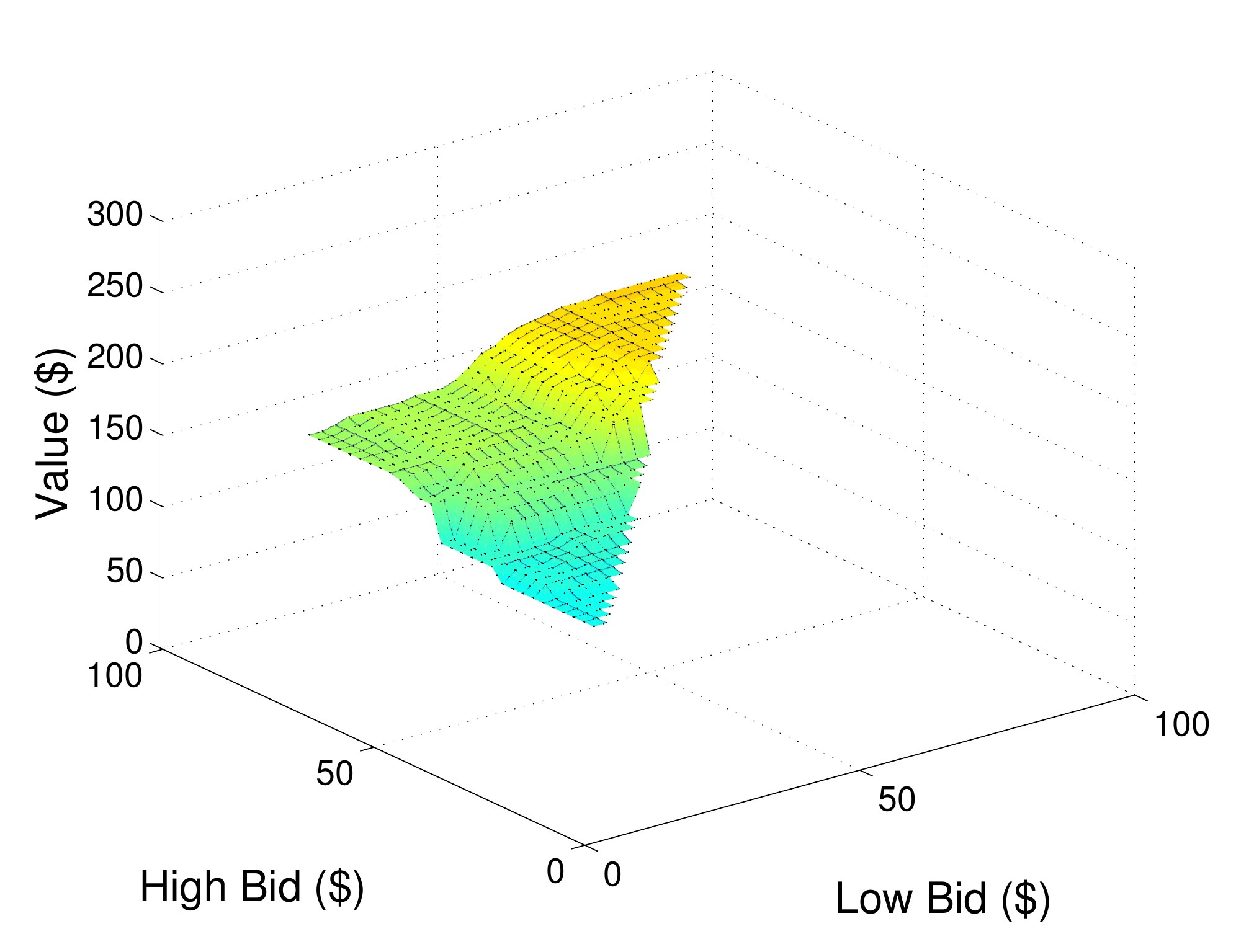}
                \caption{M--ADP, $N=1000$}\label{subfig:adp}
        \end{subfigure}
        \caption{Visual Comparison of Value Function Approximations for $t=12$ and $R_t=3$}
        \label{fig:compare}
\end{figure}

We notice that as the cardinality of the state space increases, the value of monotonicity preservation becomes more pronounced. This is especially evident in Problem $F$, where after $N=1000$ iterations, Monotone--ADP--Bidding achieves 45.9\% optimality while traditional approximate value iteration does not even reach 10\%. Although this finite state, lookup table version of approximate value iteration (for lookup table) is also a convergent algorithm (see Proposition 4.6 of \cite{Bertsekas1996}), its performance is markedly worse, especially when the state space is large. Because it exploits the monotone structure, Monotone--ADP--Bidding has the ability to quickly attain the general shape of the value function. Figure \ref{fig:earlyiter} illustrates this by showing the approximations at early iterations of the algorithm. These numerical results suggest that the convergence rate of the ADP algorithm is substantially increased through the use of the monotonicity preserving operation.

\begin{figure}[h]
        \centering
        \begin{subfigure}[b]{0.3\textwidth}
                \centering
                \includegraphics[width=\textwidth]{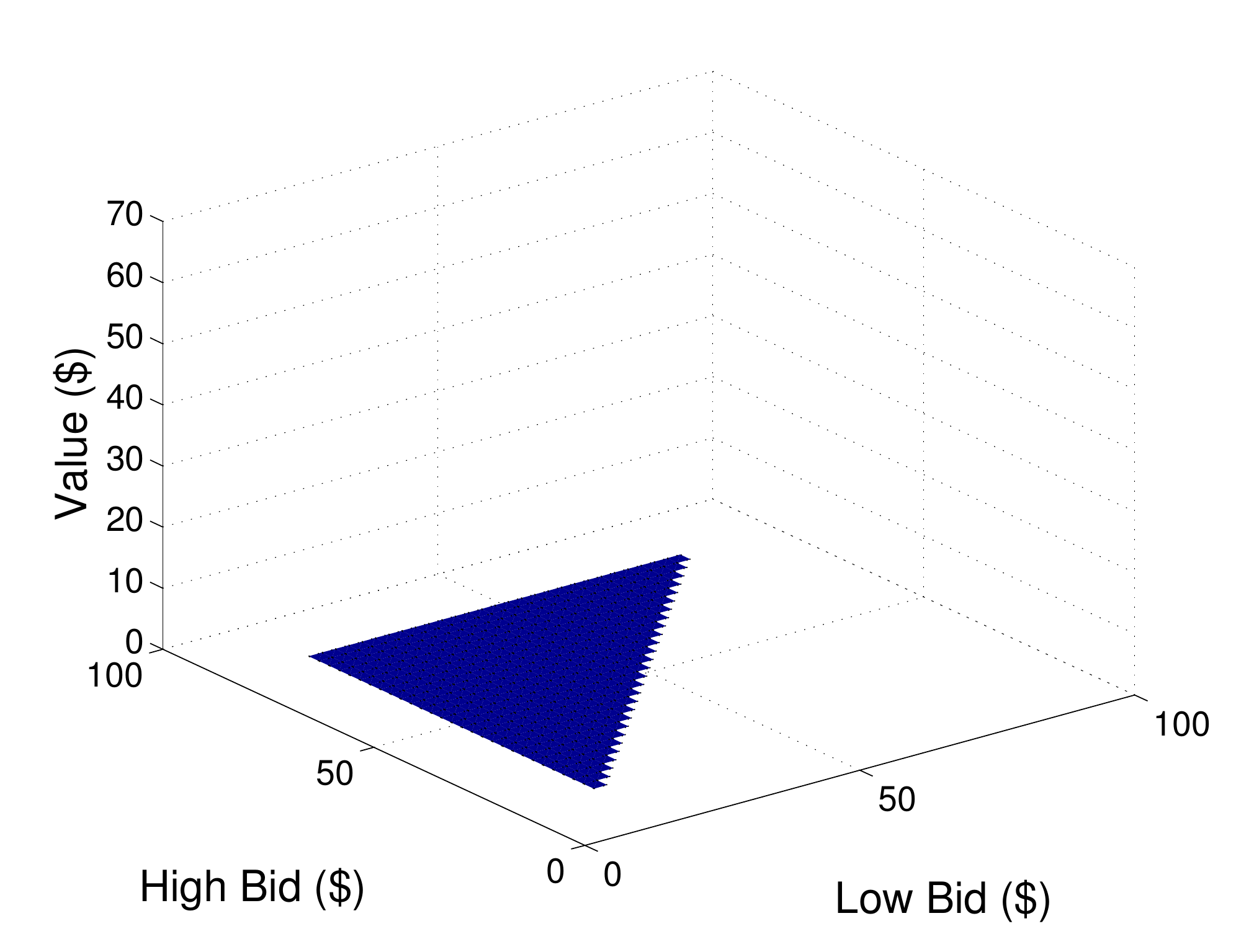}
                \caption{Iteration $n=0$}
        \end{subfigure}%
        ~ 
        \begin{subfigure}[b]{0.3\textwidth}
                \centering
                \includegraphics[width=\textwidth]{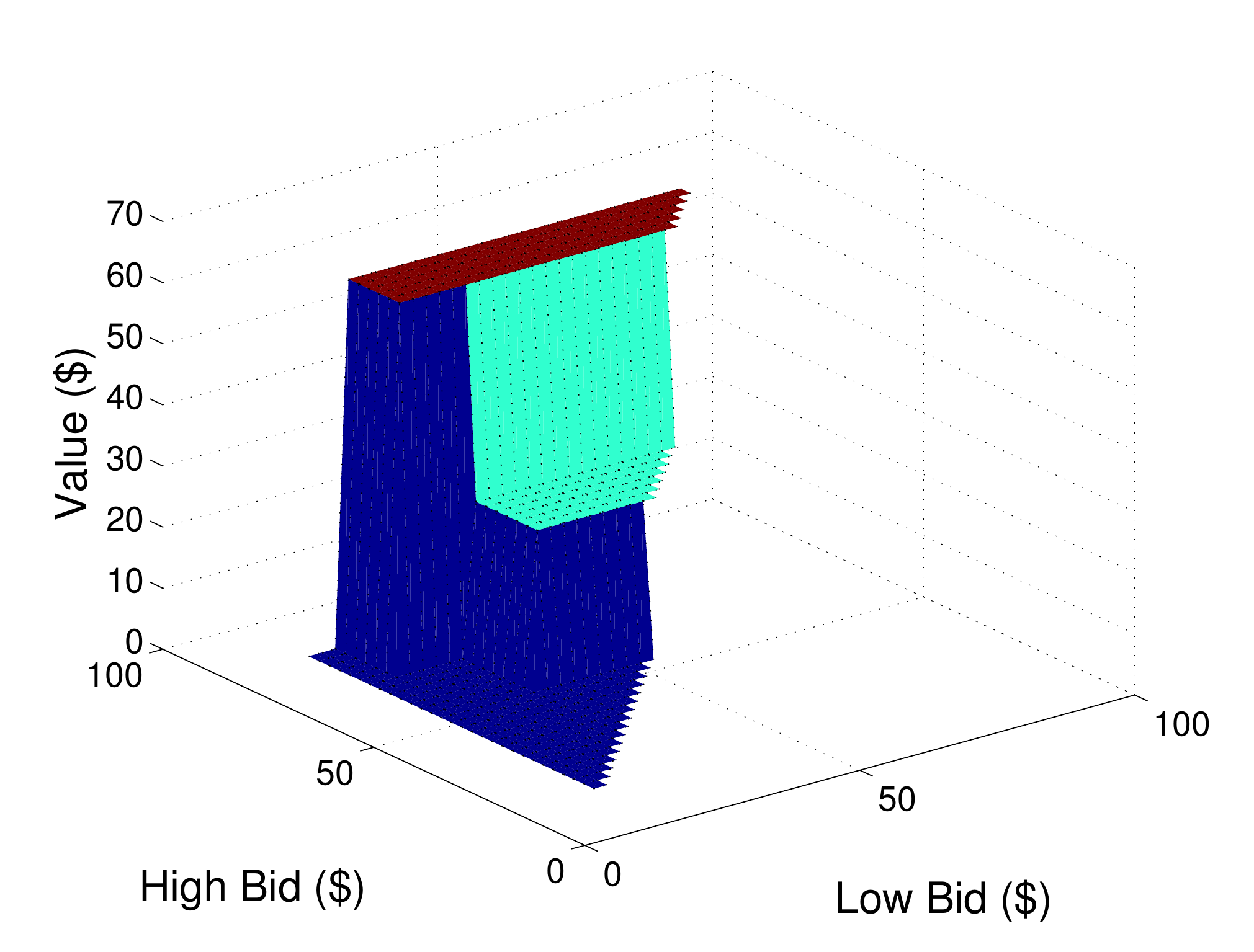}
                \caption{Iteration $n=10$}
        \end{subfigure}
        \begin{subfigure}[b]{0.3\textwidth}
                \centering
                \includegraphics[width=\textwidth]{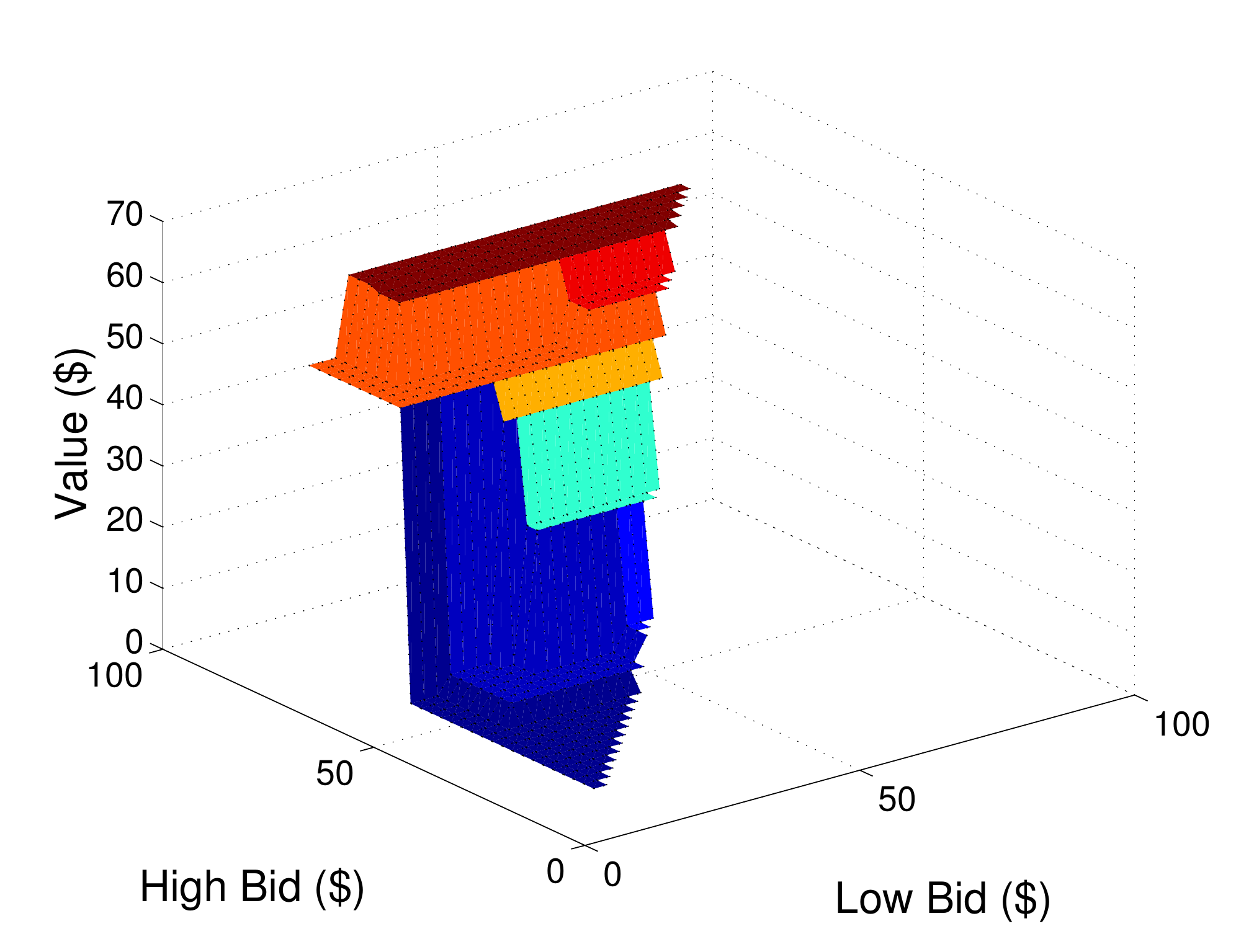}
                \caption{Iteration $n=50$}
        \end{subfigure}
        \caption{Value Function Approximations from Early Iterations of Monotone--ADP}
        \label{fig:earlyiter}
\end{figure}

\begin{table}[h]
\centering
\tiny
\begin{tabular}{@{}llcccccc@{}}\toprule
\multirow{2}{*}{\textbf{Iterations}} & \multirow{2}{*}{\textbf{Algorithm}} & \multicolumn{6}{c}{\textbf{Problem}}\\
& & $A_1$  & $B_1$ & $C_1$ & $D_1$ & $E_1$ & $F_1$\\
\midrule

\multirow{2}{*}{$N = 1000$}	&	M--ADP	&	58.9\%	&	67.8\%	&	73.5\%	&	60.7\%	&	56.8\%	&	45.9\%	\\
	&	AVI	&	53.7\%	&	45.7\%	&	66.6\%	&	23.4\%	&	24.8\%	&	7.8\%	\vspace{0.4em}\\ 
\multirow{2}{*}{$N = 5000$}	&	M--ADP	&	83.7\%	&	82.8\%	&	87.2\%	&	73.8\%	&	66.1\%	&	64.1\%	\\
	&	AVI	&	60.7\%	&	67.3\%	&	82.1\%	&	43.8\%	&	52.6\%	&	49.0\%	\vspace{0.4em}\\ 
\multirow{2}{*}{$N = 9000$}	&	M--ADP	&	89.4\%	&	93.6\%	&	93.3\%	&	76.2\%	&	74.9\%	&	86.6\%	\\
	&	AVI	&	70.2\%	&	75.8\%	&	85.3\%	&	46.7\%	&	58.8\%	&	57.3\%	\vspace{0.4em}\\ 
\multirow{2}{*}{$N = 13000$}	&	M--ADP	&	93.8\%	&	89.9\%	&	96.8\%	&	79.8\%	&	83.7\%	&	88.5\%	\\
	&	AVI	&	76.3\%	&	83.1\%	&	87.8\%	&	49.8\%	&	68.2\%	&	57.8\%	\vspace{0.4em}\\ 
\multirow{2}{*}{$N = 17000$}	&	M--ADP	&	95.8\%	&	96.4\%	&	97.8\%	&	82.7\%	&	86.8\%	&	91.4\%	\\
	&	AVI	&	78.0\%	&	85.1\%	&	90.7\%	&	62.2\%	&	72.0\%	&	70.6\%	\vspace{0.4em}\\ 
\multirow{2}{*}{$N = 21000$}	&	M--ADP	&	95.0\%	&	98.4\%	&	98.1\%	&	90.5\%	&	87.8\%	&	92.7\%	\\
	&	AVI	&	81.1\%	&	87.7\%	&	90.0\%	&	61.0\%	&	73.7\%	&	76.3\%	\vspace{0.4em}\\ 
\multirow{2}{*}{$N = 25000$}	&	M--ADP	&	97.0\%	&	98.5\%	&	98.5\%	&	89.7\%	&	90.4\%	&	94.8\%	\\
	&	AVI	&	86.4\%	&	89.4\%	&	92.1\%	&	60.0\%	&	75.1\%	&	76.0\%	\vspace{0.4em}\\ 		 
\bottomrule
\end{tabular}
\vspace{1em}
\caption{\% Optimal of Policies from the M--ADP and AVI Algorithms for Variation 1}
\label{table:comparison2}
\end{table}

With the effectiveness of Monotone--ADP--Bidding on Variation 1 established, we now examine its \emph{computational} benefits over backward dynamic programming. A comparison of CPU times between Monotone--ADP--Bidding and backward dynamic programming is shown in Figure \ref{fig:cpucompare2}, where the horizontal axis is in log--scale. Once again, we notice the order of magnitude difference in computation time for the exact solution and for the near--optimal ADP solution. Indeed from Table \ref{table:timesave2}, we see that we can achieve very good solutions using an ADP approach while cutting computational resources by over 93\%. In the most drastic case, Problem $F$ (over 150,000 states), a 95\% optimal solution is achieved using only 4\% the amount of computational power.

\begin{figure}[h]
        \centering
        \begin{subfigure}[b]{0.3\textwidth}
                \centering
                \includegraphics[width=\textwidth]{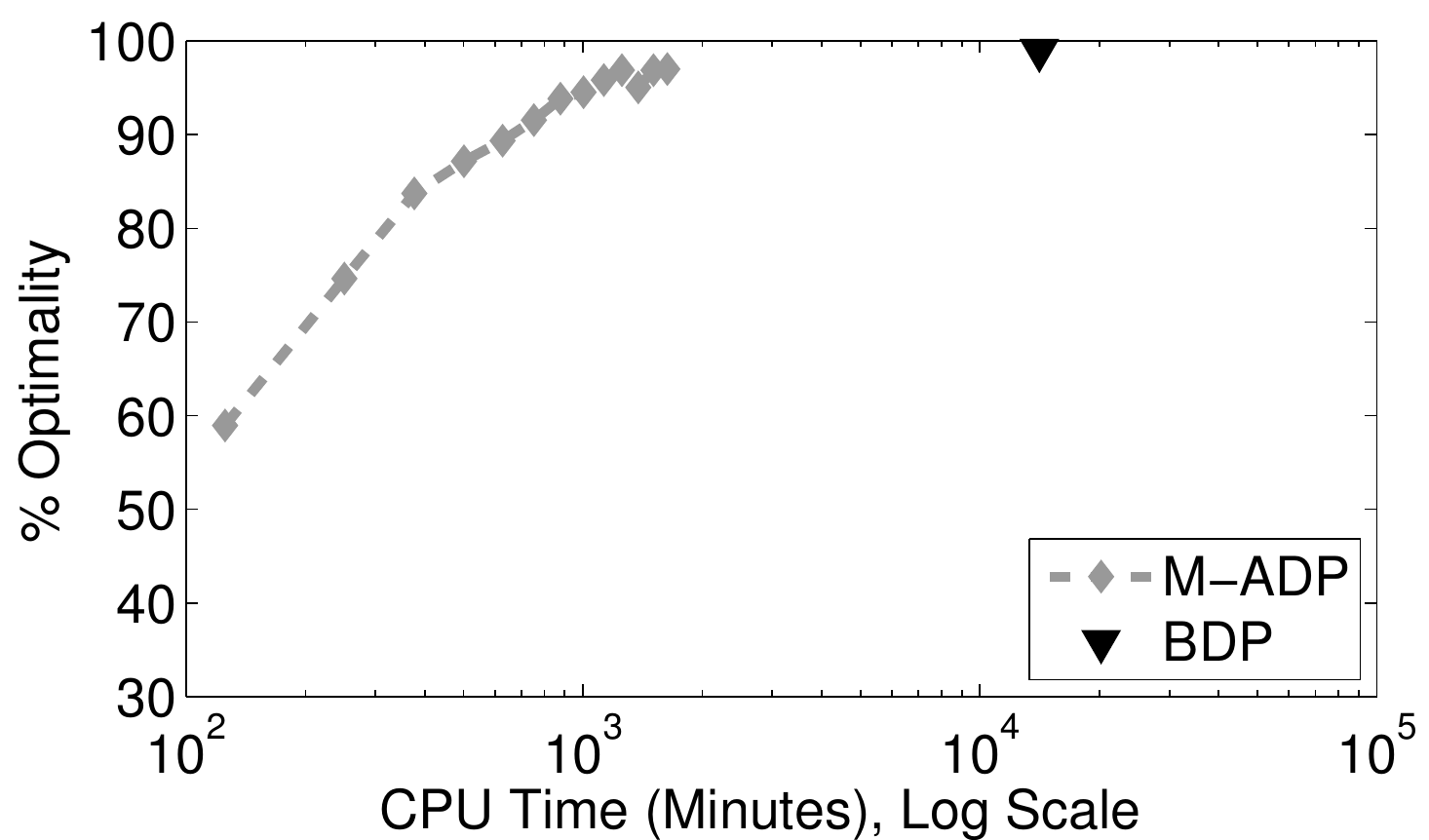}
                \caption{Problem $A_1$}
        \end{subfigure}
        \begin{subfigure}[b]{0.3\textwidth}
                \centering
                \includegraphics[width=\textwidth]{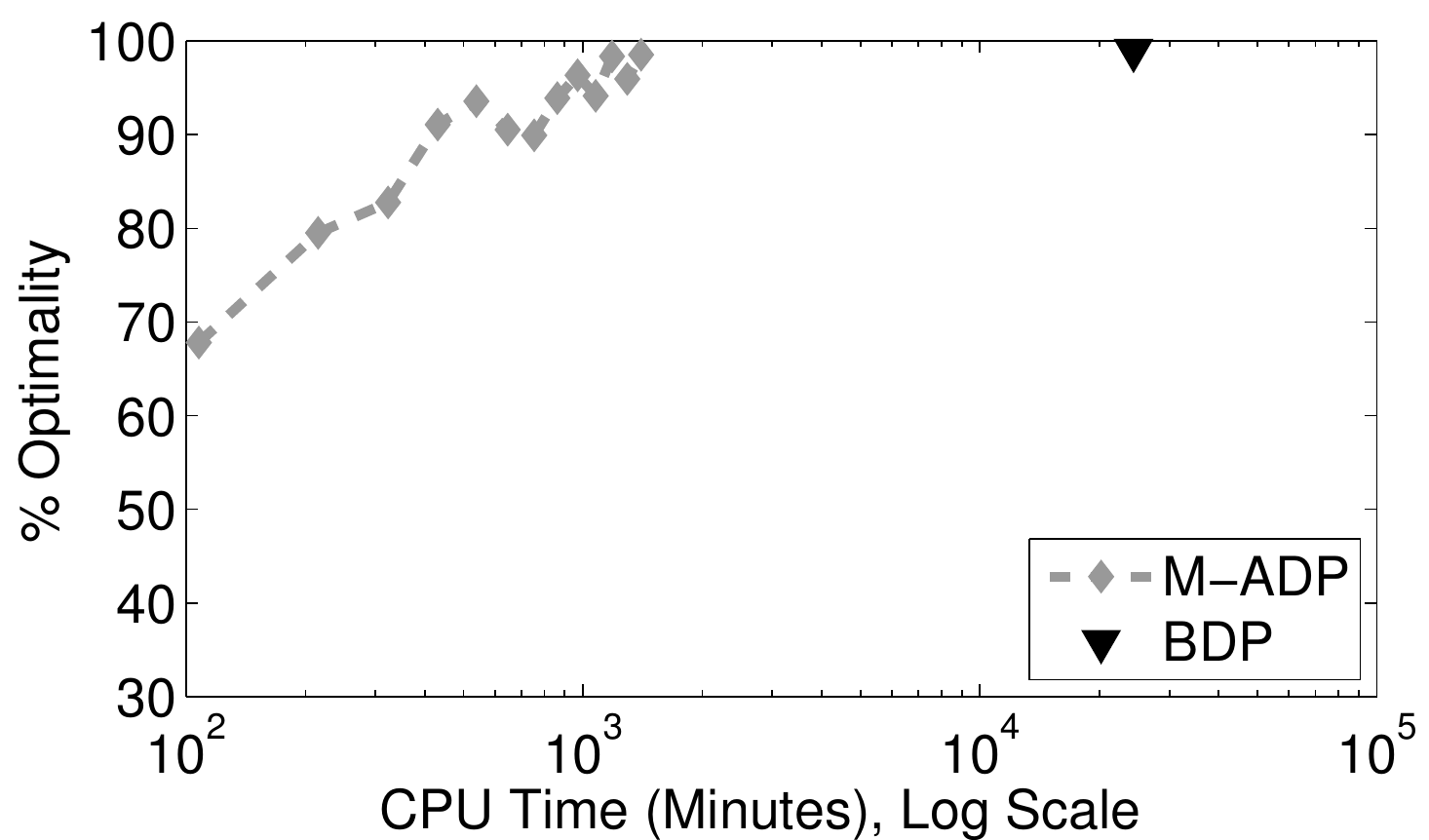}
                \caption{Problem $B_1$}
        \end{subfigure}
        \begin{subfigure}[b]{0.3\textwidth}
                \centering
                \includegraphics[width=\textwidth]{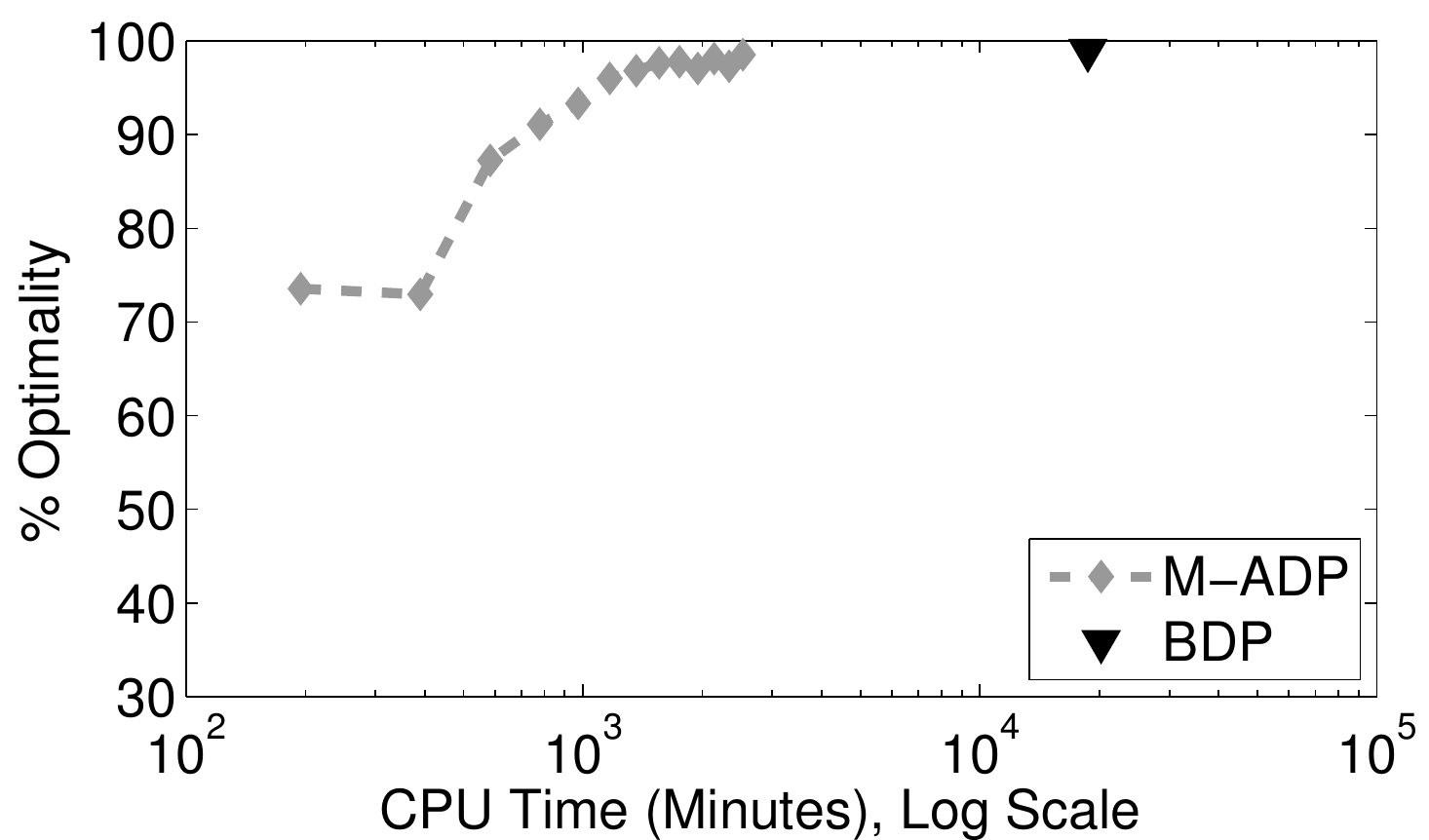}
        \caption{Problem $C_1$}
        \end{subfigure}\\
       
               \begin{subfigure}[b]{0.3\textwidth}
                \centering
                \includegraphics[width=\textwidth]{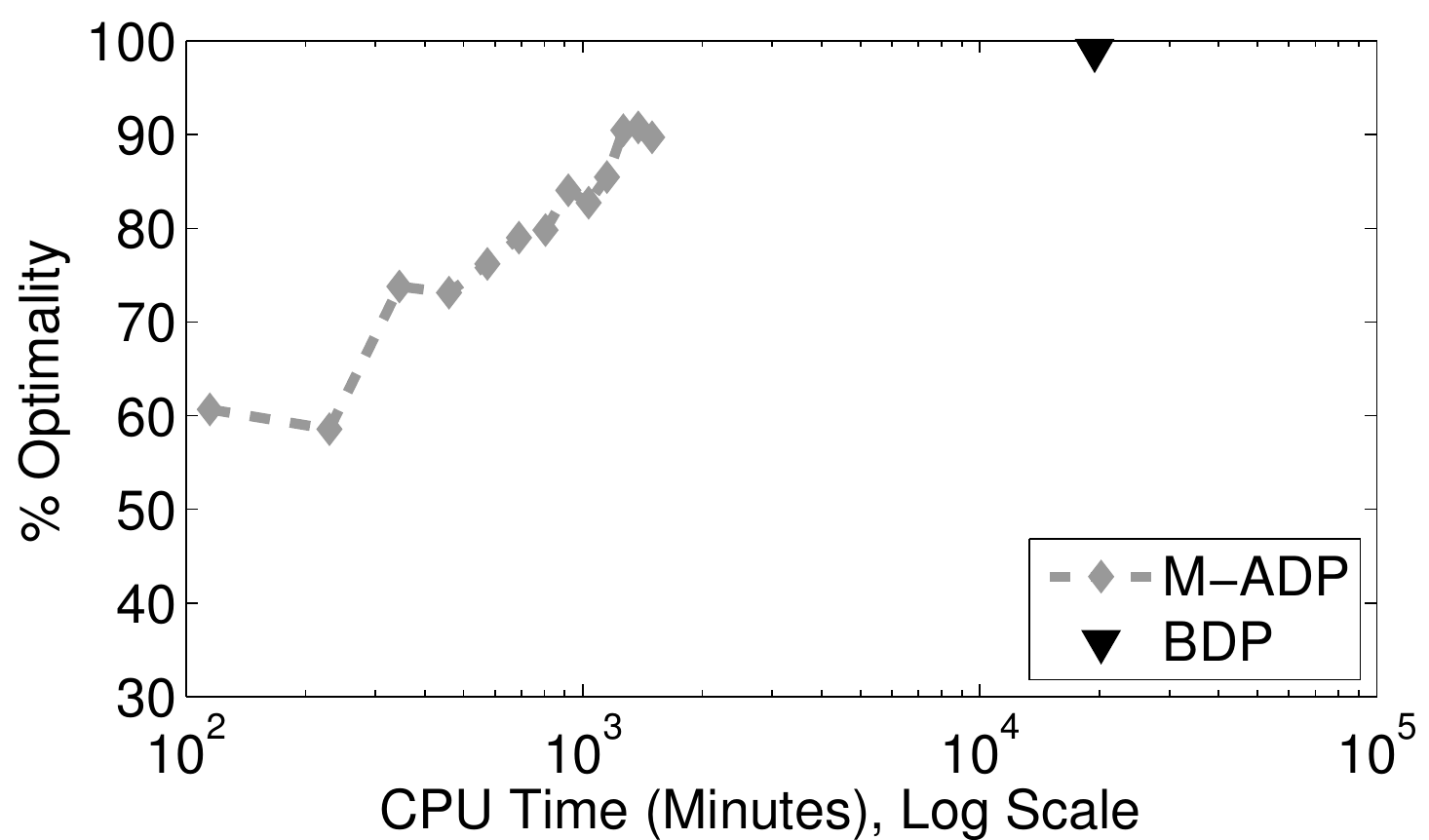}
                \caption{Problem $D_1$}
        \end{subfigure}
        \begin{subfigure}[b]{0.3\textwidth}
                \centering
                \includegraphics[width=\textwidth]{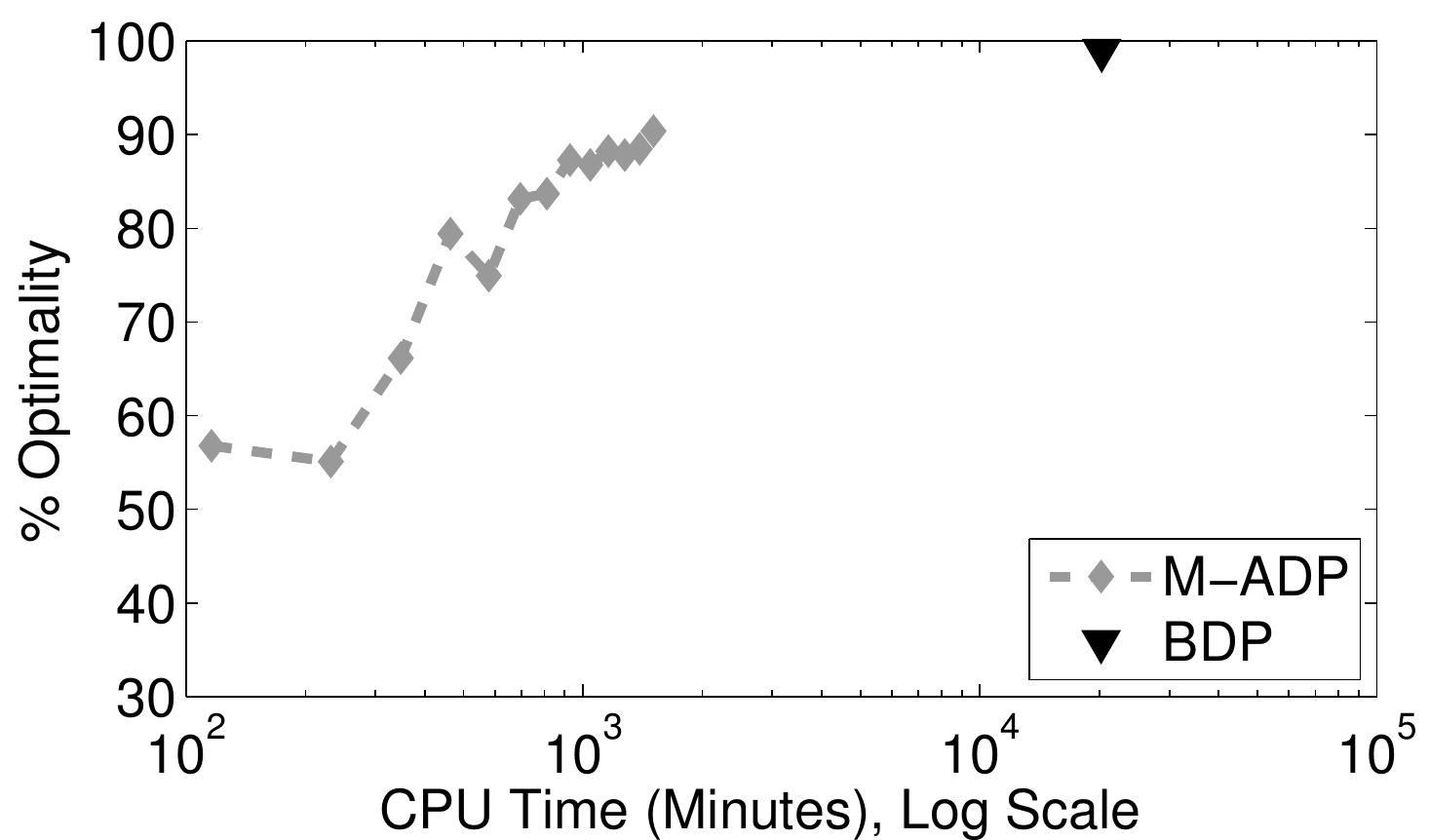}
                \caption{Problem $E_1$}
        \end{subfigure}
        \begin{subfigure}[b]{0.3\textwidth}
                \centering
                \includegraphics[width=\textwidth]{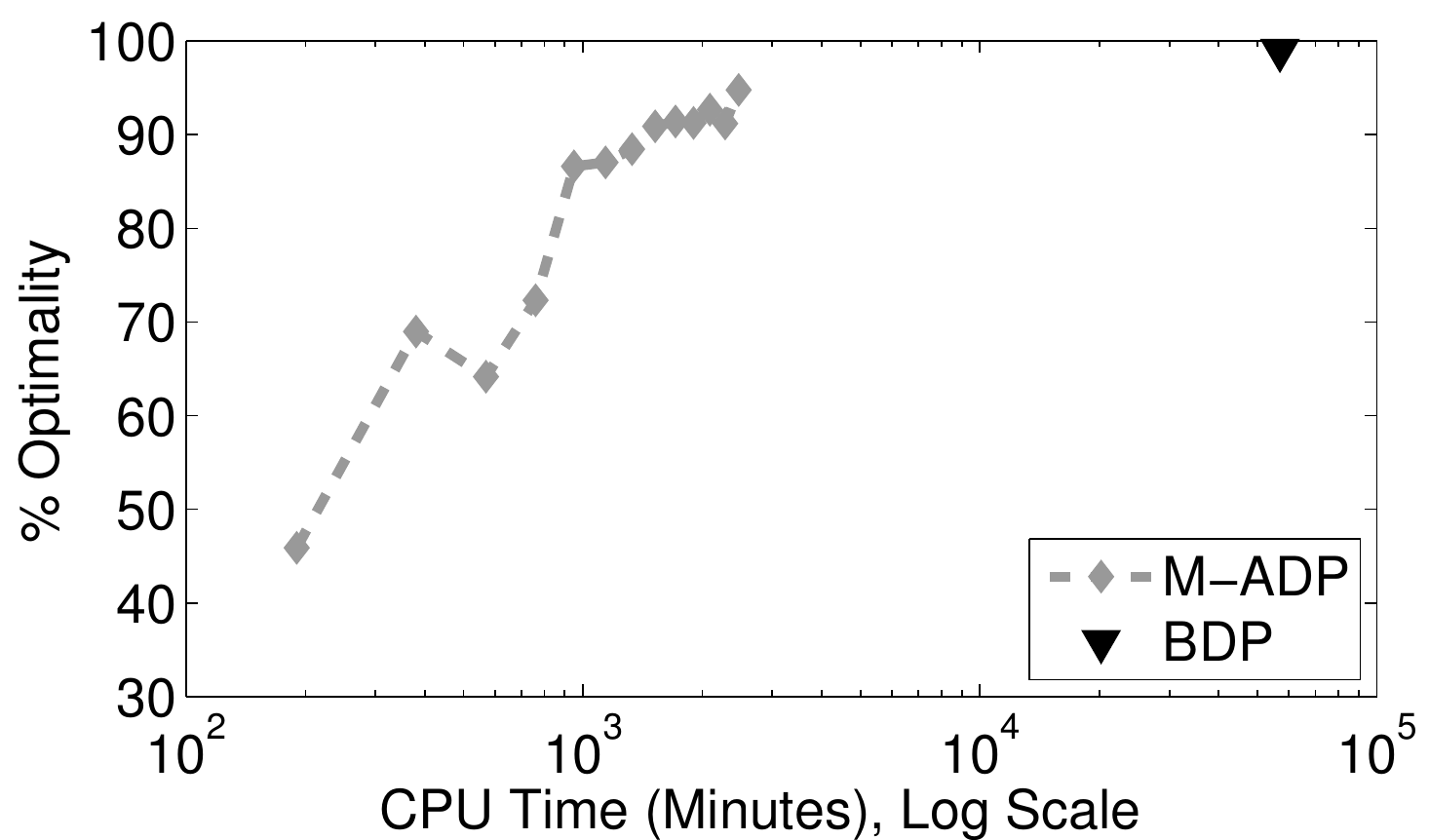}
        \caption{Problem $F_1$}
        \end{subfigure}
        \caption{Computation Times of M--ADP vs. BDP for Variation 1}
        \label{fig:cpucompare2}
\end{figure}
\begin{table}[h]
\centering
\tiny
\begin{tabular}{@{}lcccccc@{}}\toprule
\multirow{2}{*}{\textbf{}} & \multicolumn{6}{c}{\textbf{Problem}}\\
& $A_1$ & $B_1$ & $C_1$ & $D_1$ & $E_1$ & $F_1$\\
\midrule
BDP CPU Time (Minutes) &	14{,}112	&	24{,}392	&	18{,}720	&	19{,}448	& 20{,}256	&	56{,}968 \vspace{0.4em} \\

M--ADP CPU Time (Minutes)/\% Optimality  &	1{,}003/95\% &	1{,}077/94\%	&	1{,}167/96\%	&	 1{,}264/90\%	&	1{,}506/90\%	&	2{,}470/95\% \vspace{0.4em} \\

\textbf{\% Time Savings}  &	\textbf{93\%} &	\textbf{96\%}	&	\textbf{94\%}	&	 \textbf{94\%}	&	\textbf{93\%}	&	\textbf{96\%}\\			
\bottomrule
\end{tabular}
\vspace{1em}
\caption{\% Time Savings from BDP when using M--ADP Near--Optimal Solution}
\label{table:timesave2}
\end{table}

\subsection{Variation 2} Briefly, we also consider a problem with a more complex price process: a Markov Regime--Switching model with two regimes, denoted by the process $X_t$. We represent the normal regime as $X_t = 0$ and the spike regime as $X_t=1$. Let $S(t)$ be a deterministic seasonal component, $\epsilon_t$ be discrete, independent and identically distributed (i.i.d.)\ random variables representing noise in the normal regime, and $\epsilon^s_t$ be discrete, i.i.d.\ random variables representing noise in the spike regime. The price process can be written as:
\begin{equation*}
P_t = S(t) + \indicate{X_t=0} \cdot \epsilon_t + \indicate{X_t=1} \cdot \epsilon^s_t.
\end{equation*}
Also, we define the transition probabilities of the (time--inhomogenous) Markov chain $X_t$:
\begin{equation*}
p_{i,j}(t) = \mathbf{P}(X_{t+1}=j\,|\,X_t=i).
\end{equation*}
Because $X_t$ only takes two states, let $p(t) = p_{0,1}(t)$ (the probability, at time $t$, of moving from the normal regime into the spike regime) and $q(t) = p_{1,0}(t)$ (the probability, at time $t$, of returning to the normal regime). 
The state variable for this problem is five--dimensional: $S_t = (R_t, L_t, b_{t-1}^-, b_{t-1}^+, X_t)$. In order to generate a small library of test problems, we considered two versions of the seasonal component:
\[
S_i(t) = 15 \, f_i(2\pi t/12)+50,
\]
for $i \in \{1,2\}$ and $f_1(x)=\sin(x)$ and $f_2(x)=\cos(x)$.
We roughly model the fact that price spikes tend to occur more frequently when demand is high. Since demand is often modeled using sinusoidal functions, we use the following for $p(t)$ (the probability of moving from the normal regime to the spike regime) when the seasonal component is $S_i(t)$:
\begin{equation*}
p_i(t) = \alpha_p \, \big[f_i(2\pi t/12)+1\big]/2,
\end{equation*}
for some parameter $\alpha_p \le 1$, representing the maximum probability of moving to the spike regime: $p_t(t) \in [0,\alpha_p]$. In these numerical results, $q(t)$, the probability of returning to the normal regime, is always modeled as a constant $\alpha_q$. Moreover, both $\epsilon_t$ and $\epsilon_t^s$ have support $\{-10, -9, -8, \ldots, +39, +40\}$ and are distributed according to the discrete pseudonormal distribution (with parameters $(\mu_X,\sigma_X) = (0,7)$ and $(\mu_X,\sigma_X)=(15,20)$, respectively). The skewed support allows us to model the preponderance of upward spikes in electricity spot prices. 
The remainder of the parameters vary across the test problems and are summarized in Table \ref{table:problems2} below.
\begin{table}[h]
\centering
\tiny
\begin{tabular}{@{}cccccccccc@{}}\toprule
\multirow{1}{*}{\textbf{Problem}} & \multirow{1}{*}{$T$} & \multirow{1}{*}{$\Rmax$} & \multirow{1}{*}{$\Lmax$} & \multirow{1}{*}{$\beta(l)$} & \multirow{1}{*}{\textbf{Trend}} & \multirow{1}{*}{$\alpha_p$} & \multirow{1}{*}{$\alpha_q$} & \multirow{1}{*}{\textbf{Cardinality of }$\mathcal S$}  \\
\midrule
$A_2$ & $24$ & $4$ & $6$ & $(l/6)^\frac{1}{6}$ & $S_2(t)$ & 0.9 & 0.5 & 22{,}320 \\
$B_2$ & $24$ & $4$ & $8$ & $(l/8)^\frac{1}{6}$ & $S_1(t)$ & 0.8 & 0.7 & 29{,}760 \\
$C_2$ & $12$ & $8$ & $6$ & $(l/6)^\frac{1}{6}$ & $S_2(t)$ & 0.9 & 0.5 & 44{,}640 \\
$D_2$ & $12$ & $6$ & $8$ & $(l/8)^\frac{1}{6}$ & $S_2(t)$ & 0.8 & 0.7 & 44{,}640 \\
$E_2$ & $12$ & $8$ & $10$ & $(l/10)^\frac{1}{6}$ & $S_1(t)$ & 0.9 & 0.5 & 74{,}400 \\
$F_2$ & $12$ & $10$ & $8$ & $(l/8)^\frac{1}{6}$ & $S_2(t)$ & 0.8 & 0.7 & 74{,}400 \\
\bottomrule
\end{tabular}
\vspace{1em}
\caption{Parameter Choices for Variation 2 Benchmark Problem}
\label{table:problems2}
\end{table}

We ran both Monotone--ADP--Bidding and traditional approximate value iteration for 10,000 iterations on each of the test problems. The results of the benchmarking are summarized in Table \ref{table:benchmark} below (for brevity, we omit plots of the approximate value function and computation times and instead state that they are very comparable to those of Variation 1). It is clear that, once again, Monotone--ADP--Bidding provides significantly better solutions than approximate value iteration, particularly in the early iterations. 
\vspace{1em}
\begin{table}[h]
\centering
\tiny
\begin{tabular}{@{}lllllllll@{}}\toprule
\multirow{2}{*}{\textbf{Iterations}} & \multirow{2}{*}{\textbf{Algorithm}} & \multicolumn{6}{c}{\textbf{Problem}}\\
& & $A_2$  & $B_2$ & $C_2$ & $D_2$ & $E_2$ & $F_2$\\
\midrule

\multirow{2}{*}{$N = 2000$}	&	M--ADP	&	82.4\%	&	82.6\%	&	94.6\%	&	93.6\%	&	82.8\%	&	82.8\%	\\
	&	AVI	&	31.3\%	&	32.2\%	&	43.6\%	&	60.9\%	&	33.0\%	&	46.2\%	\vspace{0.4em}\\ 
\multirow{2}{*}{$N = 4000$}	&	M--ADP	&	86.7\%	&	83.7\%	&	96.1\%	&	99.1\%	&	93.2\%	&	90.0\%	\\
	&	AVI	& 53.4\%	&	46.1\%	&	62.1\%	&	76.9\%	&	54.0\%	&	62.4\%	\vspace{0.4em}\\ 
\multirow{2}{*}{$N = 6000$}	&	M--ADP	&	93.6\%	&	81.0\%	&	88.3\%	&	98.2\%	&	90.2\%	&	90.5\%	\\
	&	AVI	&	64.8\%	&	51.0\%	&	69.3\%	&	82.6\%	&	63.5\%	&	76.1\%	\vspace{0.4em}\\ 
\multirow{2}{*}{$N = 8000$}	&	M--ADP	&	95.3\%	&	86.8\%	&	92.2\%	&	93.8\%	&	93.4\%	&	88.8\%	\\
	&	AVI	&	77.4\%	&	68.0\%	&	67.5\%	&	79.6\%	&	77.0\%	&	77.0\%	\vspace{0.4em}\\ 
\multirow{2}{*}{$N =10000$}	&	M--ADP	&	94.4\%	&	87.8\%	&	95.8\%	&	96.3\%	&	95.2\%	&	98.2\%	\\
	&	AVI	&	84.1\%	&	58.7\%	&	77.9\%	&	71.8\%	&	84.3\%	&	60.8\%	\vspace{0.4em}\\ 
	 
\bottomrule
\end{tabular}
\vspace{1em}
\caption{\% Optimal of Policies from the M--ADP and AVI Algorithms for Variation 2}
\label{table:benchmark}
\end{table}

\section{Case Study: Training and Testing an ADP Policy Using Real NYISO Data}
\label{sec:casestudy}
In this section, we use the distribution--free, post--decision state version of Monotone--ADP to produce bidding policies for the New York City zone of the NYISO, with the goal of demonstrating the idea of training using only historical data as ``sample paths.'' The case study uses two full years of 5--minute real--time price data obtained from the NYISO, for the recent years of 2011 and 2012. See Figure \ref{fig:spotprices} below for a visual comparison.
\begin{figure}[h]
        \centering
        \begin{subfigure}[b]{0.45\textwidth}
                \centering
                \includegraphics[width=\textwidth]{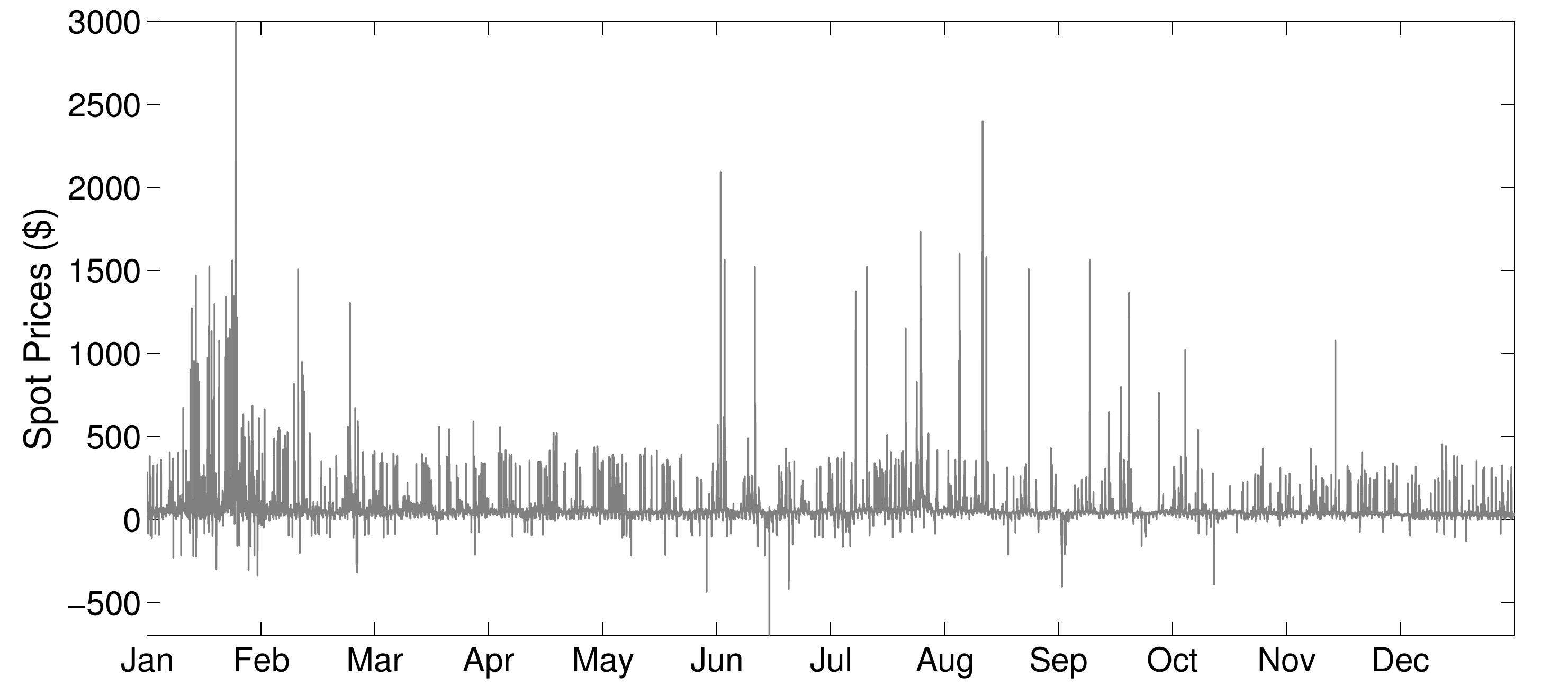}
                \caption{2011 Real--Time Prices}\label{subfig:2011}
        \label{subfig:optimalbdp}
        \end{subfigure}
        \begin{subfigure}[b]{0.45\textwidth}
                \centering
                \includegraphics[width=\textwidth]{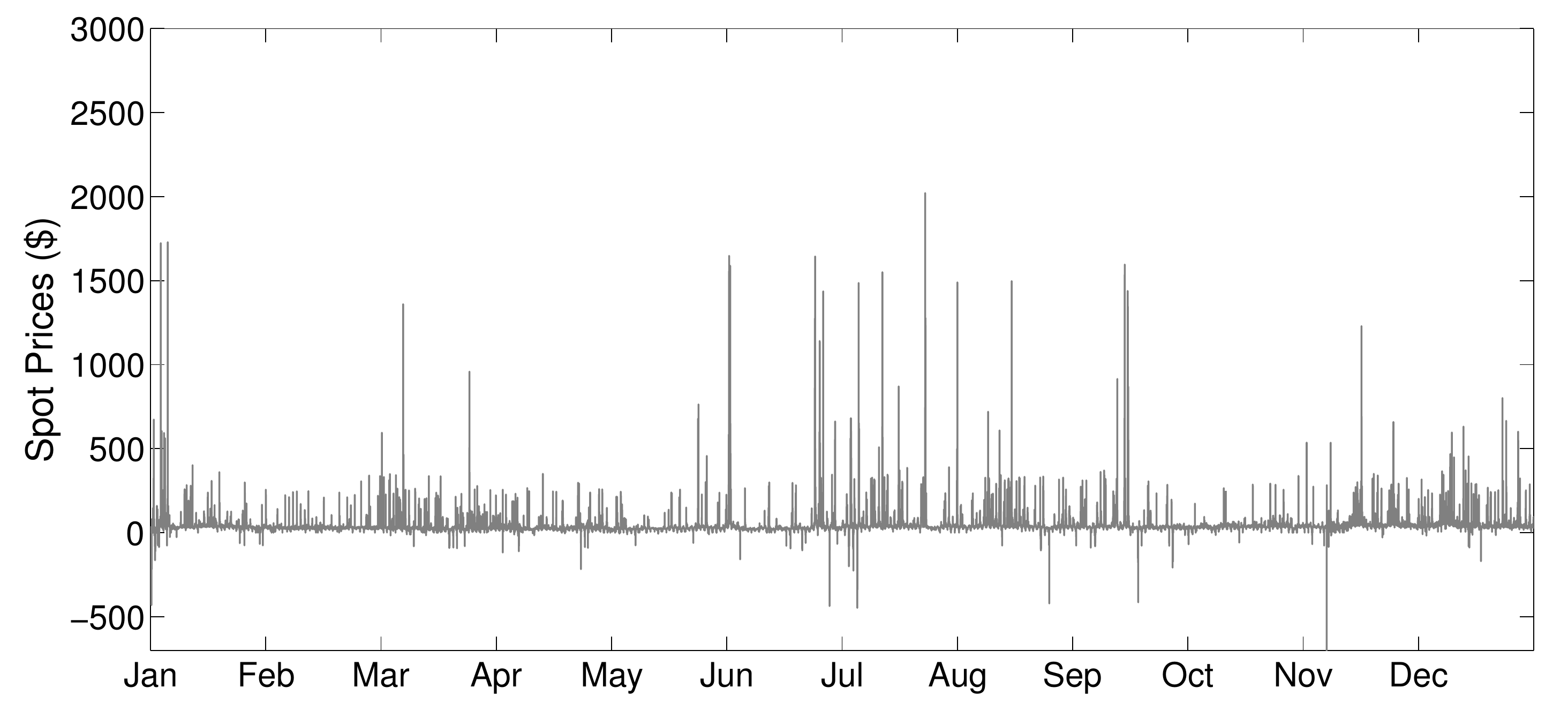}
                \caption{2012 Real--Time Prices}\label{subfig:2012}
        \end{subfigure}
        \caption{NYISO Real--Time, 5--Minute Prices Used for Training and Testing of an ADP Policy}
        \label{fig:spotprices}
\end{figure}
The concern with the historical prices is that we must satisfy Assumption \ref{ass:six}; i.e., we must assume that the unknown stochastic process $P_t$ is bounded. This is not an unreasonable assumption if we allow the bound to be high, say \$3{,}000, which is consistent with the prices in our data set. We remark again that, in order to satisfy Assumption \ref{ass:seven}, we use:
\begin{equation*}
    	\hat{v}_{t}^{b,n}(S_{t}^{b,n}) = \max_{b_{t+1} \in \mathcal B} \bigl [C_{t+1,t+3}(S^n_{t+1},b_{t+1})+\widebar{V}_{t+1}^{b,n-1}(S_{t+1}^{n},b_{t+1}) \bigr ],
\end{equation*}
in Step 2a of Figure \ref{fig:algorithm2}, where the transition from $S_{t}^{b,n}$ to $S_{t+1}^n$ is accomplished using a single sample from historical data. The remaining assumptions are satisfied for the same reasons as before. 

\subsection{ADP Policies} There are many sensible ways to choose training data for a specific operating time horizon. In this paper, we consider two commonsense methods: 1) using historical samples from the same month of the previous year to train a policy for the current month (``ADP Policy 1''), and 2) using samples from the previous month to train a policy for the current month (``ADP Policy 2''). The rationale for the first method is that the price process may behave similarly in the same month across years (though factors like weather, natural gas prices, etc, should be considered before assuming that such a statement is true), and the rationale for the second method is to simply use the most recent data available. We train an hourly bidding policy that has a horizon of one day ($T+1=24$) and the post--decision state variable for this case study is five--dimensional:
\[
S_t^b = (R_t, b_{t-1}^-, b_{t-1}^+, b_t^-, b_t^+) \in \mathcal S^b,
\]
where the bids are linearly discretized between $\bmin=0$ and $\bmax=150$ into 15 grid points in each dimension. Although it is difficult to discern from Figure \ref{fig:spotprices}, 98.2\% of the prices in our data set are below \$150. To have a lower dimensional state variable for more reasonable runtimes, we elect to assume $P_t^S = \{\}$ and $\beta(l)=1$ (it is also not typical for a battery manufacturer to provide an expression for $\beta(l)$; an accurate model for $\beta(l)$ would require estimation from empirical data, which is outside the scope of this paper). Conversations with industry colleagues suggested that, for this application, it is reasonable to model a 1 MW, 6 MWh battery. Since $M=12$, we choose $\Rmax = 72$, giving us a state space of size $|\mathcal S^b| = 3.6$ million states, much larger than that of the benchmark problems in the previous section. The remaining details are summarized in the list below.
\begin{enumerate}
\item Due to the fact that the characteristics of the spot prices can be very different on weekends (see e.g., \cite{Coulon2012}), we considered weekdays only. In a true application, it would be important to train a separate policy for weekends.
\item In order to have a larger data set for our simulations, our main assumption is that spot prices of a particular hour are identically distributed across weekdays of the same month, allowing us to train and test on a large set of sample paths.
\item We train a daily value function for each month of the year. In essence, we combine the data for the weekdays of each month to produce a policy that is valid for any given weekday of the same month.
\item As before, we set the undersupply penalty parameter $K$ to 1.
\end{enumerate}
The real--time prices from 2011 are used as training data and the prices from 2012 are used simultaneously as training and test data: for each month of 2012, we generate two policies, one trained using data from the same month in 2011 and the other trained using data from the previous month. The revenues generated by these policies are given in Table \ref{table:realdata}, where the evaluation method from Section \ref{sec:benchmarking} is used. The results correspond to running the algorithm for $N=100{,}000$ iterations. Note that because the post--decision version does not compute an expectation, each iteration is significantly faster than that of the pre--decision version, but in general, requires more iterations. The results show that ADP Policy 1 (training on data from the same month of the previous year) narrowly outperforms ADP Policy 2 (training on data from the previous month) in most cases. Although our MDP only optimizes for revenue in expectation, we nevertheless report that the (0.05--quantile, 0.95--quantile) of daily revenue for ADP Policy 1 is (\$60.37, \$474.24) with a median of \$174.55. For ADP Policy 2, we have a (0.05--quantile, 0.95--quantile) of daily revenue of (\$44.13, \$453.30) with a median of \$154.12. These results confirm that the policies consistently generate revenue.
\begin{table}[h]
\centering
\scriptsize
\begin{tabular}{@{}cccccc@{}}\toprule
\multirow{2}{*}{\textbf{Test Dataset}} & \multicolumn{2}{c}{\textbf{ADP Policy 1}} &\multicolumn{2}{c}{\textbf{ADP Policy 2}}\\
& Training Dataset & Revenue ($\$$) & Training Dataset & Revenue ($\$$)\\
\midrule
January--12 & January--11 & 6{,}539.00 & December--11 & 7{,}857.69\\
February--12 & February--11 & 1{,}966.02 & January--12 & 2{,}061.99 \\			 
March--12 & March--11 & 5{,}810.70 & February--12 & 5{,}511.07\\			 
April--12 & April--11 & 4{,}147.60 & March--12 & 4{,}223.85\\			 
May--12 & May--11 & 9{,}030.54& April--12 & 8{,}296.17\\			 
June--12 & June--11 & 11{,}465.39 & May--12 & 10{,}934.07\\			 
July--12 & July--11 & 11{,}323.50 & June--12 & 9{,}042.77 \\			 
August--12 & August--11 & 6{,}277.31 & July--12 & 6{,}206.56 \\			 
September--12 & September--11 & 5{,}754.93& August--12 & 5{,}561.24\\			 
October--12 & October--11 & 3{,}693.01 & September--12 & 3{,}623.33 \\			 			 
November--12 & November--11 & 7{,}228.85 & October--12 & 2{,}768.00 \\			 
December--12 & December--11 & 3{,}275.84 & November--12 & 3{,}160.28 \\			 
\midrule			 
\multicolumn{1}{c}{\textbf{Yearly Revenue:}} & & \textbf{76{,}512.68}& & \textbf{69{,}247.02} \\
\bottomrule
\end{tabular}
\vspace{1em}
\caption{Performance of Monotone--ADP--Bidding Policy Trained and Tested on Real Data}
\label{table:realdata}
\end{table}

\subsection{Comparison to Standard Trading Policies} This subsection compares the ADP policies to several other rule--based policies developed from both the literature and discussions with industry. Due to the existence of (and lack of access to) proprietary trading strategies, we cannot claim to be comparing against the best; however, we do believe that the basic strategies surveyed in this paper are involved in a significant portion of high performing trading/bidding policies. Trading policies $A$ and $B$ are based on determining peak and off--peak hours using historical price data and inspired by strategies considered in the paper \cite{Walawalkar2007} and the technical report \cite{Byrne2012}, but adapted to our bidding setting. Trading policy $C$ uses the idea of bidding at certain quantiles and attributed to ideas given to us by our industry colleagues. Policies subscripted by $1$ (i.e., $A_1$, $B_1$, and $C_1$) use historical price data from the same month of the previous year, and policies subscripted by $2$ use data from the previous month.
\begin{description}[labelindent=\parindent,leftmargin=\parindent]
\item[Policies $A_1$ and $A_2$:] Taking advantage of the trend that lower prices occur at night, we split the operating day hours into two intervals $1$ to $h^*$ and $h^*+1$ to 24, with $h^* > 6$. The intervals are then sorted using average historical prices. If hour $h$ of the first interval has one of the six lowest prices, then it is designated a \emph{buy} interval. Similarly, if hour $h$ of the second interval has one of the six highest prices, then it is a \emph{sell} interval. All other hours are \emph{idle} intervals. When placing a bid $b_t$, we consider the hour $h$ corresponding to $(t+1,t+2]$: if hour $h$ is a buy interval, we choose $b_t = (\bmax,\bmax)$; if hour $h$ is a sell interval, we choose $b_t = (\bmin, \bmin)$; and if hour $h$ is an idle interval, we choose $b_t=(\bmin,\bmax)$. This policy essentially guarantees (with the possible exception of spike situations where prices exceed $\bmax$) that we fill up the battery in the interval from 1 to $h^*$ and then empty it in the interval from $h^*+1$ to 24. With some tuning, we found that $h^*=12$ provided the highest valued policies.
\item[Policies $B_1$ and $B_2$:] The second set of policies are again based on the idea of pairing periods of low prices with periods of high prices, but with more flexibility than policies $A_1$ and $A_2$. Instead, we sort all hours of a given day using average historical prices and designate the $k^*$ lowest priced hours as \emph{buy} intervals, corresponding to $b_t = (\bmax,\bmax)$ and the $k^*$ highest priced hours as \emph{sell} intervals, corresponding to $b_t = (\bmin, \bmin)$. The remaining hours are \emph{idle} intervals, meaning we set $b_t = (\bmin, \bmax)$. Again using historical prices, at time $t$, we estimate the level of resource $\hat{R}_{t+1}$ at the beginning of the next hour as the average of the outcomes of $R_{t+1}$ over historical sample paths. When encountering a buy interval with $\hat{R}_{t+1} > 60$ (nearly full battery) or a sell interval with  $\hat{R}_{t+1} < 12$ (nearly empty battery), we place the idle bid instead. Finally, if we detect that we have more energy in storage than can be sold in the time left until the end of horizon, we place sell bids thereafter. We report results for the tuned parameter $k^* = 10$.
\item[Policies $C_1$ and $C_2$:] Let $\alpha < 0.5$ be the parameter to our final set of policies. For each hour $h$, we compute the empirical quantiles of the historical prices at $\alpha$ and $1-\alpha$, denoted $q_\alpha$ and $q_{(1-\alpha)}$, respectively (note the suppressed dependence on $h$). When bidding at time $t$, we again estimate $\hat{R}_{t+1}$ using historical data. For times when the battery is estimated to be nearly full, we place the bid $b_t = (\bmin,q_{(1-\alpha)})$. Similarly, if the battery is nearly empty, we place the bid $b_t = (q_\alpha, \bmax)$. For anything inbetween, we simply bid $b_t = (q_\alpha, q_{(1-\alpha)})$, with the hope of consistently buying low and selling high. We implement the same logic for when we hold more energy than the maximum that can be sold in the time remaining and initiate a sell--off. In the numerical results below, we use $\alpha = 0.1$. Smaller values of $\alpha$ correspond to the notion of reserving the battery for only the highest valued trades.   
\end{description}
The results of running policies $A_i$, $B_i$, and $C_i$ are given in Table \ref{table:industry}.
\begin{table}[h]
\centering
\scriptsize
\begin{tabular}{@{}cccccccc@{}}\toprule
\multirow{2}{*}{\textbf{Test Dataset}} & \multicolumn{6}{c}{\textbf{Revenue (\$)}} \\
& \textbf{Policy $A_1$} & \textbf{Policy $A_2$} & \textbf{Policy $B_1$} & \textbf{Policy $B_2$} & \textbf{Policy $C_1$} & \textbf{Policy $C_2$}\\
\midrule
January--12 & 3,078.56 	&	3,539.68 	&	3,182.07 	&	3,445.84 	&	1,901.89 	&	8,461.97 	\\
February--12 & 707.02 	&	404.68 	&	397.38 	&	(349.51)	&	1,503.52 	&	1,487.59 	\\
March--12 & 2,380.97 	&	2,343.57 	&	1,837.57 	&	2,154.49 	&	4,744.29 	&	6,214.73 	\\
April--12 &	 702.84 	&	1,247.13 	&	205.12 	&	1,078.62 	&	3,403.25 	&	3,412.50 	\\
May--12 & 5,855.13 	&	3,564.74 	&	4,888.52 	&	3,797.41 	&	6,944.26 	&	5,013.73 	\\
June--12 & 3,449.75 	&	4,742.00 	&	4,511.81 	&	3,427.11 	&	7,329.25 	&	7,618.00 	\\			 
July--12 & 6,871.67 	&	4,488.28 	&	6,940.68 	&	6,781.36 	&	8,003.43 	&	(7,066.45)	\\
August--12 & 1,278.66 	&	1,482.63 	&	1,824.57 	&	1,273.28 	&	4,724.14 	&	4,908.08 	\\
September--12 & 1,438.39 	&	1,638.63 	&	315.94 	&	1,665.22 	&	3,868.75 	&	4,336.50 	\\
October--12 & 701.91 	&	751.93 	&	633.58 	&	321.80 	&	2,879.64 	&	2,750.99 	\\
November--12 & 1,585.50 	&	1,938.98 	&	1,354.96 	&	1,359.01 	&	4,438.00 	&	(1,270.90)	\\
December--12 & 1,240.97 	&	1,012.26 	&	424.56 	&	431.05 	&	2,703.46 	&	2,445.46 	\\
\midrule			 
\textbf{Yearly Revenue:} & \textbf{29,291.36} 	&	\textbf{27,154.52} 	&	\textbf{26,516.76} 	&	\textbf{25,385.68} 	&	\textbf{52,443.88} 	&	\textbf{38,312.20} 	\\
\bottomrule
\end{tabular}
\vspace{1em}
\caption{Performance of Standard Trading Policies Trained and Tested on Real Data}
\label{table:industry}
\end{table}
Given that they were afforded more nuanced actions than simply buy and sell, perhaps not surprisingly, Policies $C_i$ outperformed the rest. However, we also notice that, unlike the other policies, Policy $C_2$ generated large negative revenues in July--12 and November--12. Comparing Policy $C_1$ against ADP Policy 1 and comparing Policy $C_2$ against ADP Policy 2, we see the revenues generated are still a disappointing $68.5\%$ and $55.3\%$, respectively, of the ADP revenues, suggesting that it is difficult, even after tuning, for simple rule--based heuristics to perform at the level of a well--trained ADP policy that considers downstream value. Moreover, the months of July--12 and November--12 (during which Policy $C_2$ posted negative revenues) suggest that the ADP strategy is more robust to the differences in training data when compared to Policy $C_i$. A possible driving force behind Policy $C_2$'s failure to generate revenue during these months is that the training data from June--12 and October--12 has largely differing characteristics (e.g., many spikes) from the testing data in July--12 and November--12 (see Figure \ref{fig:spotprices}).
\subsection{Additional Insights}
\label{sec:additionalinsights}
Applying Monotone--ADP--Bidding to real data from the NYISO has given us several insights into the topic of energy arbitrage. First, we note that for both ADP Policy 1 and ADP Policy 2 (see Table \ref{table:realdata}), the largest revenues were generated in the months of May, June, July, presumably due to changes in weather. The difference between the revenues generated in the months of highest and lowest revenue, June and February, is more drastic than one might expect: $\textnormal{Jun Revenue} - \textnormal{Feb Revenue} = \$9{,}499.37$ for ADP Policy 1 and  $\textnormal{Jun Revenue} - \textnormal{Feb Revenue} = \$8{,}872.08$ for ADP Policy 2. These results suggest that perhaps energy arbitrage should not be a year--round investment, but rather one that is active only during months with potential for high revenue. As \cite{Sioshansi2009} concludes, when it comes to the value of energy storage, it is important to consider various revenue sources.

Costs of energy storage can be as low as \$160 kWh$^{-1}$ today, and it is reasonable to expect that they will continue to decrease. As mentioned earlier, with optimal storage control strategies and decreased capital costs, energy arbitrage can soon become profitable on its own, but as it currently stands, storage costs are still relatively high compared to potential revenue. Therefore, it is also imperative that the precise storage needs of our trading/bidding policies are well--understood; it may be the case that in some months, one would choose to dedicate the entire battery to frequency regulation, while in high revenue months, the better strategy may be to use some of the capacity toward arbitrage. It is clear that some policies, such as Policies $A_i$, are designed with fully utilizing the available storage in mind, but for more complex policies such as those generated by Monotone--ADP--Bidding, the usage profiles are not obvious. Figure \ref{fig:rdists} shows the empirical distribution for the storage level of the battery (on an hourly basis) throughout the test data set. Note that for presentation purposes we have scaled the plot so that the bar at $R_t=0$ is cut off; due to its designation as the initial state (and final state as well for most sample paths), its probability is skewed to 0.09 and 0.10, for the two plots respectively. The high probability at 1 MWh is likely explained by the fact that it corresponds to full hourly charge, the minimum amount of energy needed to avoid the possibility of an undersupply penalty.
\begin{figure}[h]
        \centering
        \begin{subfigure}[b]{0.45\textwidth}
                \centering
                \includegraphics[width=\textwidth]{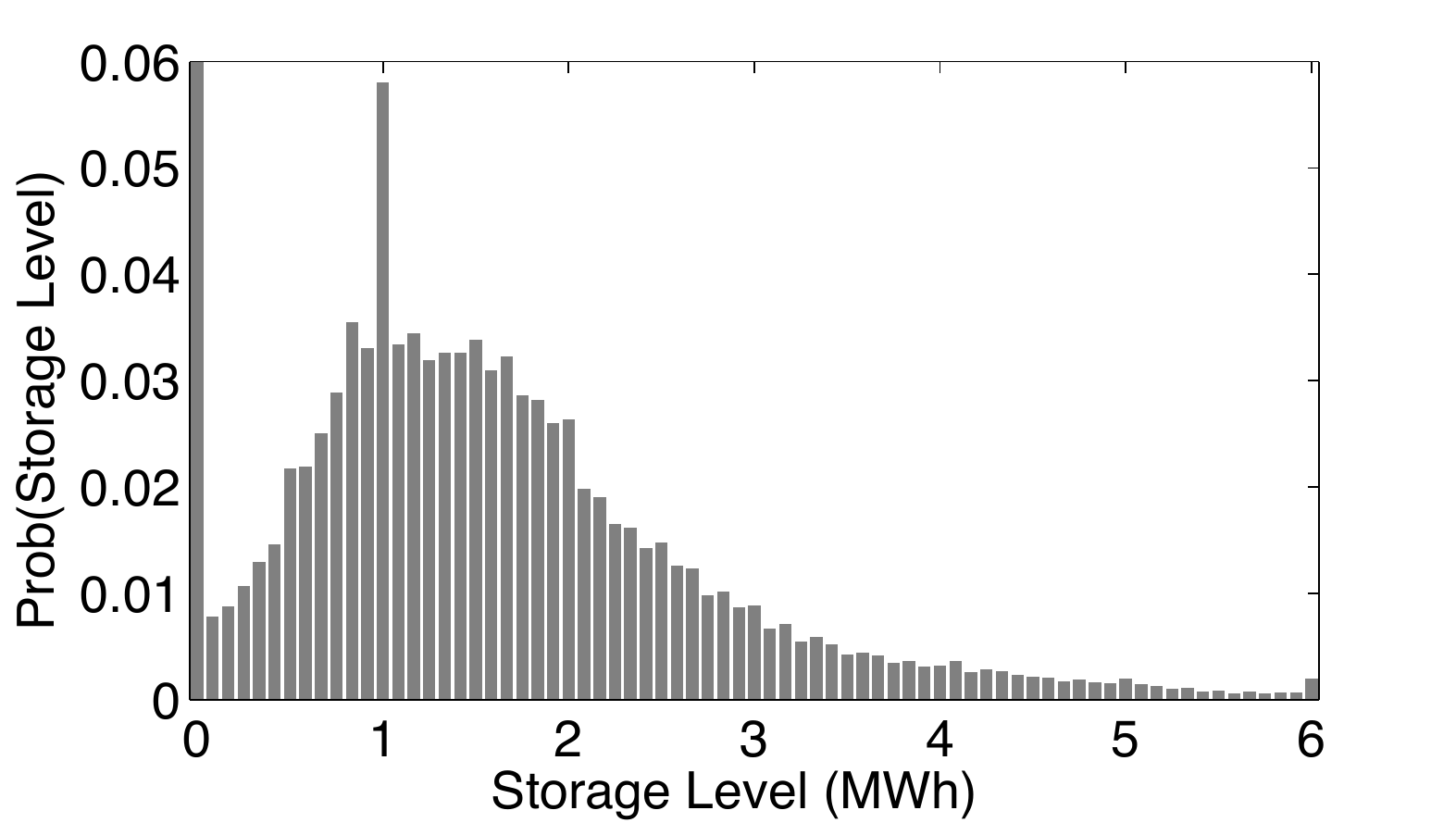}
                \caption{ADP Policy 1}\label{subfig:rdistyr}
        \label{subfig:optimalbdp}
        \end{subfigure}
        \begin{subfigure}[b]{0.45\textwidth}
                \centering
                \includegraphics[width=\textwidth]{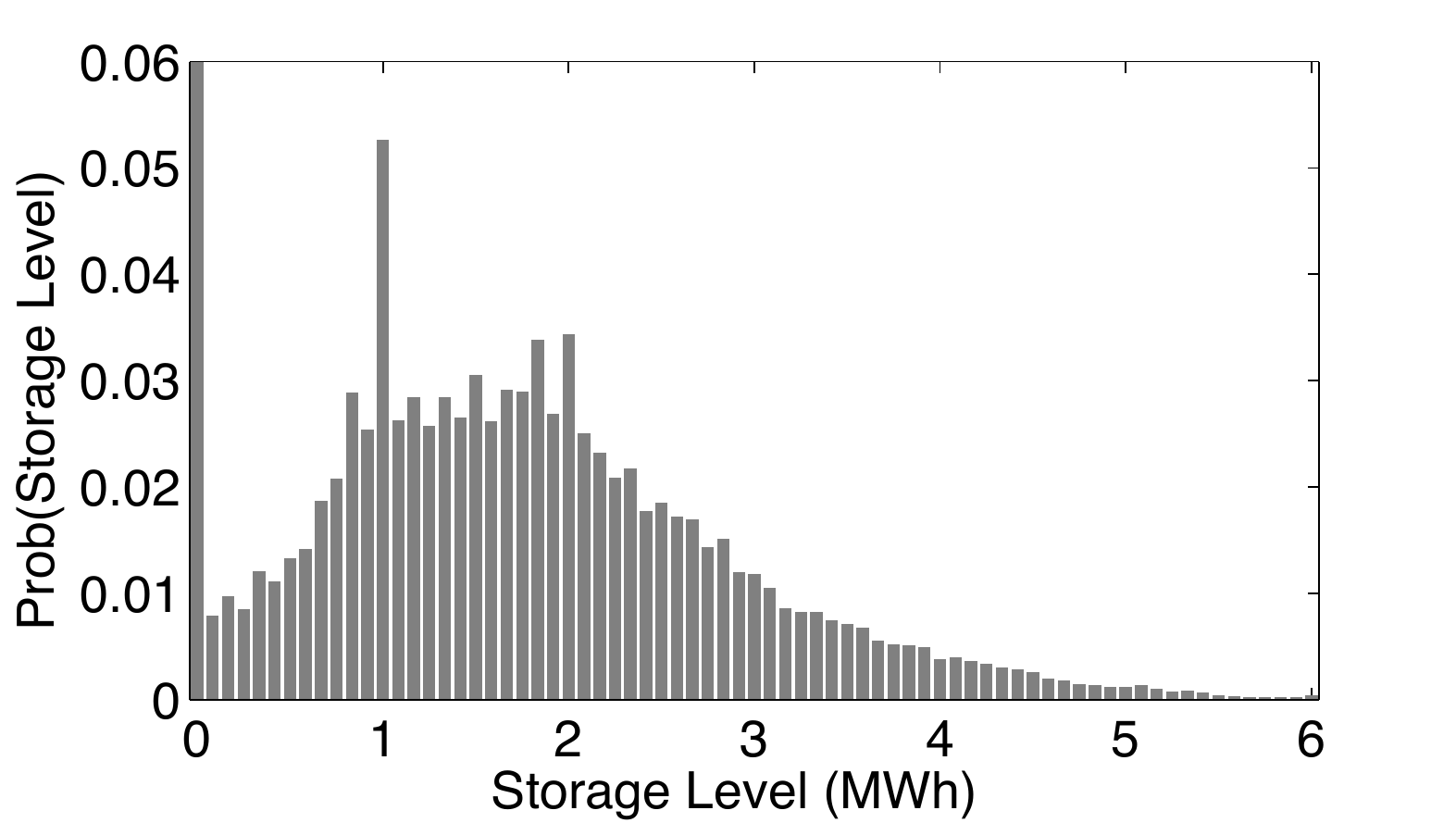}
                \caption{ADP Policy 2}\label{subfig:rdistmo}
        \end{subfigure}
        \caption{Empirical Distribution of Storage Level using ADP Policies 1 and 2}
        \label{fig:rdists}
\end{figure}
The 0.9-- and 0.95--quantiles for ADP Policy 1 occur at 3.00 MWh and 3.75 MWh, and for ADP Policy 2, they are 3.16 MWh and 3.75 MWh. This means for our (relatively short) daily trading horizon, a 6 MWh battery is unnecessary---a 33\% smaller device with 4 MWh storage would have sufficed and delivered similar results at a steep discount in capital cost. However, if a longer trading horizon, say, weekly (allowing us to take into account the low prices on the weekends), is desired, it would be necessary to train a policy using a sufficiently large battery and then using simulation to determine the effective amount of storage needed by the policy. In summary, with today's substantial capital costs, it would be prudent to do an analysis of a policy's storage needs.

Lastly, we discuss the issue of choosing $C_\textnormal{term}(s)$ in a practical implementation of the algorithm. Due to the daily cycles present in the real--time market, there is likely to be little additional value added in expending computational resources towards developing a bidding policy whose horizon lasts much longer than a few days or a week. In fact, from our conversations with industry colleagues, we envision that a bidding policy such as ours have a daily horizon that is used repeatedly day after day, with the policy retrained periodically (perhaps weekly). For such a usage scenario, it is important to correctly choose $C_\textnormal{term}(s)$, because leftover energy has value that can be capitalized on even after the true horizon of the policy. We suggest the following practical methods for determining the functional form of $C_\textnormal{term}(s)$:
\begin{enumerate}
\item Given the knowledge that the same policy is to be reused, in an effort to prevent the forced ``sell--off'' type behavior that is expected when $C_\textnormal{term}(s)=0$, it is reasonable to choose $C_\textnormal{term}(s)$ to structurally resemble $V^*_0(s)$ (i.e., up to constant shifts). One strategy for accomplishing this is to first compute $V^*_0(s)$ using a zero terminal contribution, and then re--solving the dynamic program using the previously computed $V^*_0$ as the terminal contribution. This process can be iterated until the resulting policies (not the value functions themselves) are observed to converge. Our (informal) implementation of this procedure shows that the desired behavior of not forcing the storage to zero at the end of the time horizon is indeed attained.
\item After training an initial policy, we can determine, by inspecting the resource paths, a point in time where the storage level is empty or very low (e.g., immediately after a period of high prices). The horizon of the problem can then be redefined so that $T$ corresponds to this point in time and a new policy (with zero terminal contribution) can be trained. Essentially, we hope that the forced sell--off is translated to a point in time where a natural sell--off would have likely occurred.
\end{enumerate}

\section{Conclusion}
In this paper, we describe an hour--ahead bidding and battery arbitrage problem for a real--time electricity market (e.g. NYISO's real--time market). We then formulate the problem mathematically as an MDP and show that the optimal value function satisfies a monotonicity property, a structural result that can be exploited in order to accelerate the convergence of ADP algorithms. The algorithm that we employ is called Monotone--ADP--Bidding and uses monotonicity to infer the value of states nearby an observed state. When benchmarked against a traditional approximate value iteration algorithm, we found that the improvements in terms of solution quality were drastic. Furthermore, the ADP algorithm can reach near--optimal solutions without the need for significant computational time and power (which an exact solution technique like backward dynamic programming certainly requires); in fact, our empirical results show that near--optimal solutions can be generated using less than 10\% of the computational resources necessary for backward dynamic programming. We also describe and sketch the proof of convergence for a distribution--free method where we can train value functions with Monotone--ADP--Bidding using historical spot prices --- this removes the need for us to perform the difficult task of specifying and fitting an accurate stochastic model of spot prices. In our case study, the method is tested on two large data sets: the 5--minute real--time prices from the NYISO from the years of 2011 and 2012. The policies from Monotone--ADP--Bidding help us conclude that energy arbitrage may be most valuable if practiced in a select few, high revenue months. Finally, the ADP policies consistently generated more revenue than several rule--based heuristic strategies that we considered, confirming that an ADP approach that approximates future value is worthwhile. 
 
\clearpage
\appendix
\section{Proofs}
\label{sec:appendix}
\gmono*
\begin{proof}
Since \[
g^R_{m+1}(R_t,q_s) = \bigl[\min \{g^R_m(R_t,q_s)-e_m^\intercal q_s, \Rmax\}\bigr]^+,
\]
it is clear that the transition from $g_m^R$ to $g_{m+1}^R$ is nondecreasing in the value of $g_m$ and nonincreasing in the value of $e_m^\intercal q_s$. Thus, a simple induction argument shows that for $r_1, r_2 \in \mathcal R$ and $q_1, q_2 \in \{-1,0,1\}^M$ where $r_1 \le r_2$ and $q_1 \le q_2$,
\begin{equation*}
g^R_M(r_1,q_2) \le g^R_M(r_2,q_1).
\end{equation*}
The result follows from the fact that $q(P,b)$ is nonincreasing in $b$.
\end{proof}

\glmono*
\begin{proof}
The transition
\[
g^L_{m+1}(L_t,d_s) = \bigl[g^L_m(L_t,d_s)-e_m^\intercal d_s\bigr]^+
\]
is nondecreasing in $g_m^L$ and nonincreasing in $e_m^\intercal d_s$. Like in Proposition \ref{gmono}, induction shows that for $l_1, l_2 \in \mathcal L$ and $d_1, d_2 \in \{0,1\}^M$ where $l_1 \le l_2$ and $d_1 \le d_2$,
\begin{equation*}
g^L_M(l_1,d_2) \le g^L_M(l_2,d_1).
\end{equation*}
The result follows from the fact that $d(P,b)$ is nonincreasing in $b$. 
\end{proof}

\Ctmonoprop*
\begin{proof}
First, we argue that the revenue function $C(r,l,P,b)$ is nondecreasing in $r$ and $l$. From their respective definitions, we can see that $\gamma_m$ and $U_m$ are both nondecreasing in their first arguments. These arguments can be written in terms of $r$ and $l$ through the transition functions $g_m^R$ and $g_m^L$. Applying Proposition \ref{gmono} and Proposition \ref{glmono}, we can confirm that $C(r,l,P,b)$ is nondecreasing in $r$ and $l$. By its definition,
\begin{equation*}
C_{t,t+2}(S_t,b_t) = \mathbf{E}\Bigl[C\bigl( g^R(R_t,P_{(t,t+1]},b_{t-1}), g^L(L_t,P_{(t,t+1]},b_{t-1}), P_{(t+1,t+2]}, b_t \bigr) \,|\, S_t\Bigr].
\end{equation*}
Again, applying Proposition \ref{gmono} and Proposition \ref{glmono} (for $m=M$), we see that the term inside the expectation is nondecreasing in $R_t$, $b_{t-1}^-$, and $b_{t-1}^+$ (composition of nondecreasing functions) for any outcome of $P_{(t,t+1]}$ and $P_{(t+1,t+2]}$. Thus, the expectation itself is nondecreasing.
\end{proof}

\Vtmonoprop*
\begin{proof}
Define the function $V_t^b(S_t,b_t) = \mathbf{E}\bigl[V^*_{t+1}(S_{t+1})\,|\,S_t,b_t\bigr]$, often called the \emph{post--decision} value function (see \cite{Powell2011}). Thus, we can rewrite the optimality equation as:
\begin{equation}
\begin{aligned}
&V^*_t(S_t) =\max_{b_t \in \mathcal B} \bigl [C_{t,t+2}(S_t,b_t)+V_t^b(S_t,b_t) \bigr ] \text{ for } t=0,1,2,\ldots,T-1,\\
&V^*_{T}(S_{T}) = C_\textnormal{term}(S_T).
\end{aligned}
\label{bellmanpost}
\end{equation}
 The proof is by backward induction on $t$. The base case is $t=T$ and since $V^*_T(\cdot)$ satisfies monotonicity for any state $s \in \mathcal S$ by definition. Notice that the state transition function satisfies the following property. Suppose we have a fixed action $b_t$ and two states $S_t =(R_t, L_t, b_{t-1}, P_t^S)$ and $S_t'=(R_t', L_t', b_{t-1}', P_t^S)$ where $(R_t, L_t, b_{t-1}) \le (R_t', L_t', b_{t-1}')$. Then, for any realization of the intra--hour prices $P_{(t,t+1]}$ (by Propositions \ref{gmono} and \ref{glmono}),
 \begin{align}
 S_{t+1} &= (R_{t+1},L_{t+1},b_t,P_{t+1}^S)=S^M(S_t,b_t,P_{(t,t+1]}),\nonumber\\  
 S_{t+1}' &= (R_{t+1}',L_{t+1}',b_t,P_{t+1}^S)=S^M(S_t',b_t,P_{(t,t+1]}),\nonumber
 \end{align}
 with $R_{t+1} \le R_{t+1}'$ and $L_{t+1} \le L_{t+1}'$, implying that $S_{t+1} \le S_{t+1}'$. This means that the transition function satisfies a specialized nondecreasing property. Using this and supposing that $V_{t+1}^*(\cdot)$ satisfies the statement of the proposition (induction hypothesis), it is clear that $V_t^b(S_t,b_t)$ is nondecreasing in $R_t$, $L_t$, and $b_{t-1}$. Now, by the previous proposition, we see that the term inside the maximum of (\ref{bellmanpost}) is nondecreasing in $R_t$, $L_t$, and $b_{t-1}$ for any action $b_t$. Hence, we can take the pointwise maximum and retain monotonicity; the inductive step is complete.
\end{proof}

\Vtpostmonoprop*
\begin{proof}
Previously in the proof of Proposition \ref{prop:Vtmono}, we argued that $V_t^b(S_t^b)$ is monotone in $R_t$, $L_t$, $b_{t-1}^-$, and $b_{t-1}^+$. To see the monotonicity in $b_t^-$ and $b_t^+$, first fix an outcome of $P_{(t,t+1]}$ and $b_t, b_t' \in \mathcal B$, with $b_t \le b_t'$. Observe that if we let $(R_{t+1},L_{t+1},b_t,P_{t+1}^S) = S^M(S_t,b_t,P_{(t,t+1]})$, then $(R_{t+1},L_{t+1},b'_t,P_{t+1}^S) = S^M(S_t,b_t',P_{(t,t+1]})$, with only the bid dimensions changed. Therefore,
\[
S^M(S_t,b_t,P_{(t,t+1]}) \le S^M(S_t,b_t',P_{(t,t+1]}).
\]
Thus, by Proposition \ref{prop:Vtmono}, for a fixed $S_t$, any outcome of the price process $P_{(t,t+1]}$, and $b_t \le b_t'$,
\begin{equation*}
V_{t+1}^*\bigl(S^M(S_t,b_t,P_{(t,t+1]})\bigr) \le V_{t+1}^*\bigl(S^M(S_t,b_t',P_{(t,t+1]})\bigr).
\end{equation*}
Hence, after taking expectations, we get the desired result: $V_t^b(S_t,b_t) \le V_t^b(S_t,b_t')$.
\end{proof}

\lemone*
\begin{proof}
We first show that $H$ satisfies the following properties:
\begin{enumerate}[label=(\roman*),labelindent=1in]
\item $V \le V' \Longrightarrow HV \le HV'$.
\item $V^*$ is a unique fixed point of $H$, i.e., $HV^* = V^*$.
\item $HV-\eta e \le H(V-\eta e) \le H(V+\eta e) \le HV+\eta e$, for $\eta > 0$.
\end{enumerate}
Statement (i) is trivial and follows directly from the monotonicity of the max and expectation operators. Statement (ii) follows from the fact that the finite horizon dynamic program exhibits a unique optimal value function (and thus, post--decision value function as well) determined by the backward recursive Bellman equations. Statement (iii) is easy to see directly from the definition of $H$. Now, applying Lemma 4.6 of \cite{Bertsekas1996} gives us the desired limit result.
\end{proof}

\lemtwo*
\begin{proof}
To show this, first note that given a fixed $t \le T-2$ and \emph{any} vector $Y \in \mathbb R^{|\mathcal S^b|}$ (defined over the post--decision state space) that satisfies the monotonicity property, it is true that the vector $h_t Y$, whose component at $s \in \mathcal S^b$ is defined using the post--decision Bellman recursion,
\[
(h_tY)(s) = \mathbf{E} \Bigl[\max_{b_{t+1} \in \mathcal B} \bigl [C_{t+1,t+3}(S_{t+1},b_{t+1})+Y(S_{t+1}^b) \bigr ] \,|\, S_{t}^b = s \Bigr ],
\]
also obeys the monotonicity property. We point out that there is a small difference between the operator $H$ and $h_t$ in that $H$ operates on vectors of dimension $T\cdot |\mathcal S^b|$. To verify monotonicity, $s_1, s_2 \in \mathcal S^b$ such that $s_1 \mless^b s_2$. For a fixed sample path of prices $P$, let $S_{t+1}(s_1,P)$ and $S_{t+1}(s_2,P)$ be the respective downstream pre--decision states. Applying Propositions \ref{gmono} and \ref{glmono}, we have that $S_{t+1}(s_1,P) \mless^b S_{t+1}(s_2,P)$. For any fixed $b_{t+1} \in \mathcal B$, we apply the monotonicity of the contribution function $C_{t+1,t+3}$ (Proposition \ref{prop:Ctmono}) and the monotonicity of $Y$ to see that
\begin{align}
C_{t+1,t+3}(S_{t+1}(s_1,P),b_{t+1})&+Y((S_{t+1}(s_1,P),b_{t+1}))\\
& \le C_{t+1,t+3}(S_{t+1}(s_2,P),b_{t+1})+Y((S_{t+1}(s_2,P),b_{t+1})),
\end{align}
which confirms that $(h_tY)(s_1) \le (h_tY)(s_2)$. When $t = T-1$, we set $(h_tY)(s) =  \mathbf{E}\bigl[ C_\textnormal{term}(S_{t+1}) \,|\, S_t^b=s\bigr]$ and the same monotonicity result holds.

Now, we can easily proceed by induction on $k$, noting that $U^0$ and $L^0$ satisfy monotonicity for each $t$. Assuming that $U^k$ satisfies monotonicity, we can argue that $U^{k+1}$ does as well; we first note that for any $t$, by the definition of $U^{k+1}$,
\[
U^{k+1}_t = \frac{U_t^k + \bigl(HU^k\bigr)_t}{2} = \frac{U_t^k+\bigl(h_t U^k_{t+1}\bigr)}{2}.
\]
By the induction hypothesis and the property of $h_t$ proved above, it is clear that $U_t^{k+1}$ also satisfies monotonicity and the proof is complete.
\end{proof}

\clearpage
\bibliographystyle{abbrvnat}
\bibliography{/Users/drjiang/Documents/Dropbox/Princeton/Princeton_Research/Bibtex/Bib}

\begin{thebibliography}{41}
\providecommand{\natexlab}[1]{#1}
\providecommand{\url}[1]{\texttt{#1}}
\expandafter\ifx\csname urlstyle\endcsname\relax
  \providecommand{\doi}[1]{doi: #1}\else
  \providecommand{\doi}{doi: \begingroup \urlstyle{rm}\Url}\fi

\bibitem[Barnhart et~al.(2013)Barnhart, Dale, Brandt, and Benson]{Barnhart2013}
C.~J. Barnhart, M.~Dale, A.~R. Brandt, and S.~M. Benson.
\newblock {The energetic implications of curtailing versus storing solar- and
  wind-generated electricity}.
\newblock \emph{Energy \& Environmental Science}, 6\penalty0 (10):\penalty0
  2804--2810, 2013.

\bibitem[Bellman(1957)]{Bellman1957a}
R.~E. Bellman.
\newblock \emph{{Dynamic Programming}}.
\newblock Princeton University Press, Princeton, NJ, USA, 1957.

\bibitem[Bertsekas and Tsitsiklis(1996)]{Bertsekas1996}
D.~P. Bertsekas and J.~N. Tsitsiklis.
\newblock \emph{{Neuro--Dynamic Programming}}.
\newblock Athena Scientific, Belmont, MA, 1996.

\bibitem[Breiman(1992)]{Breiman1992}
L.~Breiman.
\newblock \emph{{Probability}}.
\newblock Society of Industrial and Applied Mathematics, Philadelphia, PA,
  1992.

\bibitem[Byrne and Silva-Monroy(2012)]{Byrne2012}
R.~H. Byrne and C.~A. Silva-Monroy.
\newblock {Estimating the maximum potential revenue for grid connected
  electricity storage: Arbitrage and regulation}.
\newblock \emph{Tech. Rep. SAND2012-3863, Sandia National Laboratories}, 2012.

\bibitem[Carmona and Coulon(2014)]{Carmona2012}
R.~Carmona and M.~Coulon.
\newblock {A survey of commodity markets and structural models for electricity
  prices}.
\newblock In \emph{Quantitative Energy Finance}, pages 41--83. Springer, New
  York, 2014.

\bibitem[Carmona and Ludkovski(2010)]{Carmona2010}
R.~Carmona and M.~Ludkovski.
\newblock {Valuation of energy storage: An optimal switching approach}.
\newblock \emph{Quantitative Finance}, 10\penalty0 (4):\penalty0 359--374,
  2010.

\bibitem[Cartea and Figueroa(2005)]{Cartea2005}
A.~Cartea and M.~G. Figueroa.
\newblock {Pricing in electricity markets: A mean reverting jump diffusion
  model with seasonality}.
\newblock \emph{Applied Mathematical Finance}, 12\penalty0 (4):\penalty0
  313--335, 2005.

\bibitem[Conejo et~al.(2002)Conejo, Nogales, and Arroyo]{Conejo2002}
A.~J. Conejo, F.~J. Nogales, and J.~M. Arroyo.
\newblock {Price-taker bidding strategy under price uncertainty}.
\newblock \emph{IEEE Transactions on Power Systems}, 17\penalty0 (4):\penalty0
  1081--1088, 2002.

\bibitem[Coulon et~al.(2013)Coulon, Powell, and Sircar]{Coulon2012}
M.~Coulon, W.~B. Powell, and R.~Sircar.
\newblock {A model for hedging load and price risk in the texas electricity
  market}.
\newblock \emph{Energy Economics}, 40\penalty0 (0):\penalty0 976--988, 2013.

\bibitem[David(1993)]{David1993}
A.~K. David.
\newblock {Competitive bidding in electricity supply}.
\newblock \emph{Generation, Transmission and Distribution, IEE Proceedings C},
  140\penalty0 (5):\penalty0 421, 1993.

\bibitem[Eydeland and Wolyniec(2003)]{Eydeland2003}
A.~Eydeland and K.~Wolyniec.
\newblock \emph{{Energy and Power Risk Management}}.
\newblock Wiley, Hoboken, New Jersey, 2003.

\bibitem[George and Powell(2006)]{George2006}
A.~P. George and W.~B. Powell.
\newblock {Adaptive stepsizes for recursive estimation with applications in
  approximate dynamic programming}.
\newblock \emph{Machine Learning}, 65\penalty0 (1):\penalty0 167--198, 2006.

\bibitem[Godfrey and Powell(2001)]{Godfrey2001}
G.~A. Godfrey and W.~B. Powell.
\newblock {An adaptive, distribution-free algorithm for the newsvendor problem
  with censored demands, with applications to inventory and distribution}.
\newblock \emph{Management Science}, 47\penalty0 (8):\penalty0 1101--1112,
  2001.

\bibitem[Greenblatt et~al.(2007)Greenblatt, Succar, Denkenberger, Williams, and
  Socolow]{Greenblatt2007}
J.~B. Greenblatt, S.~Succar, D.~C. Denkenberger, R.~H. Williams, and R.~H.
  Socolow.
\newblock {Baseload wind energy: Modeling the competition between gas turbines
  and compressed air energy storage for supplemental generation}.
\newblock \emph{Energy Policy}, 35\penalty0 (3):\penalty0 1474--1492, 2007.

\bibitem[Gross and Finlay(2000)]{Gross2000}
G.~Gross and D.~Finlay.
\newblock {Generation supply bidding in perfectly competitive electricity
  markets}.
\newblock \emph{Computational \& Mathematical Organization Theory}, 6\penalty0
  (1):\penalty0 83--98, 2000.

\bibitem[Harris(2011)]{Harris2006}
C.~Harris.
\newblock \emph{{Electricity Markets: Pricing, Structures and Economics}}.
\newblock John Wiley \& Sons, 2011.

\bibitem[Jiang and Powell(2015)]{Jiang2013}
D.~R. Jiang and W.~B. Powell.
\newblock {An approximate dynamic programming algorithm for monotone value
  functions}.
\newblock \emph{arXiv preprint arXiv:1401.1590}, 2015.

\bibitem[Kim and Powell(2011)]{Kim2011}
J.~H. Kim and W.~B. Powell.
\newblock {Optimal energy commitments with storage and intermittent supply}.
\newblock \emph{Operations Research}, 59\penalty0 (6):\penalty0 1347--1360,
  2011.

\bibitem[Kleywegt et~al.(2002)Kleywegt, Shapiro, and Homem-de
  Mello]{Kleywegt2002}
A.~J. Kleywegt, A.~Shapiro, and T.~Homem-de Mello.
\newblock {The sample average approximation method for stochastic discrete
  optimization}.
\newblock \emph{SIAM Journal on Optimization}, 12\penalty0 (2):\penalty0
  479--502, 2002.

\bibitem[Lai et~al.(2010)Lai, Margot, and Secomandi]{Lai2010}
G.~Lai, F.~Margot, and N.~Secomandi.
\newblock {An approximate dynamic programming approach to benchmark
  practice-based heuristics for natural gas storage valuation}.
\newblock \emph{Operations Research}, 58\penalty0 (3):\penalty0 564--582, 2010.

\bibitem[L\"{o}hndorf and Minner(2010)]{Lohndorf2010}
N.~L\"{o}hndorf and S.~Minner.
\newblock {Optimal day-ahead trading and storage of renewable energies — An
  approximate dynamic programming approach}.
\newblock \emph{Energy Systems}, 1\penalty0 (1):\penalty0 61--77, 2010.

\bibitem[L\"{o}hndorf et~al.(2013)L\"{o}hndorf, Wozabal, and
  Minner]{Lohndorf2014}
N.~L\"{o}hndorf, D.~Wozabal, and S.~Minner.
\newblock {Optimizing trading decisions for hydro storage systems using
  approximate dual dynamic programming}.
\newblock \emph{Operations Research}, 61\penalty0 (4):\penalty0 810--823, 2013.

\bibitem[Nandalal and Bogardi(2007)]{Nandalal2007}
K.~D.~W. Nandalal and J.~J. Bogardi.
\newblock \emph{{Dynamic Programming Based Operation of Reservoirs:
  Applicability and Limits}}.
\newblock Cambridge University Press, New York, 2007.

\bibitem[Nascimento and Powell(2009)]{Nascimento2009a}
J.~M. Nascimento and W.~B. Powell.
\newblock {An optimal approximate dynamic programming algorithm for the lagged
  asset acquisition problem}.
\newblock \emph{Mathematics of Operations Research}, 34\penalty0 (1):\penalty0
  210--237, 2009.

\bibitem[Paatero and Lund(2005)]{Paatero2005}
J.~V. Paatero and P.~D. Lund.
\newblock {Effect of energy storage on variations in wind power}.
\newblock \emph{Wind Energy}, 8\penalty0 (4):\penalty0 421--441, 2005.

\bibitem[Papadaki and Powell(2003)]{Papadaki2003}
K.~P. Papadaki and W.~B. Powell.
\newblock {An adaptive dynamic programming algorithm for a stochastic
  multiproduct batch dispatch problem}.
\newblock \emph{Naval Research Logistics}, 50\penalty0 (7):\penalty0 742--769,
  2003.

\bibitem[Powell(2011)]{Powell2011}
W.~B. Powell.
\newblock \emph{{Approximate Dynamic Programming: Solving the Curses of
  Dimensionality}}.
\newblock Wiley, 2nd edition, 2011.

\bibitem[Powell et~al.(2004)Powell, Ruszczynski, and Topaloglu]{Powell2004}
W.~B. Powell, A.~Ruszczynski, and H.~Topaloglu.
\newblock {Learning algorithms for separable approximations of discrete
  stochastic optimization problems}.
\newblock \emph{Mathematics of Operations Research}, 29\penalty0 (4):\penalty0
  814--836, 2004.

\bibitem[Schwartz(1997)]{Schwartz1997}
E.~S. Schwartz.
\newblock {The stochastic behavior of commodity prices: Implications for
  valuation and hedging}.
\newblock \emph{The Journal of Finance}, 52\penalty0 (3):\penalty0 923--973,
  1997.

\bibitem[Secomandi(2010)]{Secomandi2010}
N.~Secomandi.
\newblock {Optimal commodity trading with a capacitated storage asset}.
\newblock \emph{Management Science}, 56\penalty0 (3):\penalty0 449--467, 2010.

\bibitem[Shahidehpour et~al.(2002)Shahidehpour, Yamin, and
  Li]{Shahidehpour2002}
M.~Shahidehpour, H.~Yamin, and Z.~Li.
\newblock \emph{{Market Operations in Electric Power Systems}}.
\newblock New York, 2002.

\bibitem[Sioshansi(2011)]{Sioshansi2011}
R.~Sioshansi.
\newblock {Increasing the value of wind with energy storage}.
\newblock \emph{Energy Journal}, 32\penalty0 (2):\penalty0 1--29, 2011.

\bibitem[Sioshansi et~al.(2009)Sioshansi, Denholm, Jenkin, and
  Weiss]{Sioshansi2009}
R.~Sioshansi, P.~Denholm, T.~Jenkin, and J.~Weiss.
\newblock {Estimating the value of electricity storage in PJM: Arbitrage and
  some welfare effects}.
\newblock \emph{Energy Economics}, 31\penalty0 (2):\penalty0 269--277, 2009.

\bibitem[Sioshansi et~al.(2011)Sioshansi, Denholm, and Jenkin]{Sioshansi2011a}
R.~Sioshansi, P.~Denholm, and T.~Jenkin.
\newblock {A comparative analysis of the value of pure and hybrid electricity
  storage}.
\newblock \emph{Energy Economics}, 33\penalty0 (1):\penalty0 56--66, 2011.

\bibitem[Thompson et~al.(2009)Thompson, Davison, and Rasmussen]{Thompson2009}
M.~Thompson, M.~Davison, and H.~Rasmussen.
\newblock {Natural gas storage valuation and optimization: A real options
  application}.
\newblock \emph{Naval Research Logistics}, 56\penalty0 (3):\penalty0 226--238,
  2009.

\bibitem[Topaloglu and Powell(2003)]{Topaloglu2003}
H.~Topaloglu and W.~B. Powell.
\newblock {An algorithm for approximating piecewise linear concave functions
  from sample gradients}.
\newblock \emph{Operations Research Letters}, 31\penalty0 (1):\penalty0 66--76,
  2003.

\bibitem[Tsitsiklis(1994)]{Tsitsiklis1994a}
J.~N. Tsitsiklis.
\newblock {Asynchronous stochastic approximation and Q-learning}.
\newblock \emph{Machine Learning}, 16\penalty0 (3):\penalty0 185--202, 1994.

\bibitem[Walawalkar et~al.(2007)Walawalkar, Apt, and Mancini]{Walawalkar2007}
R.~Walawalkar, J.~Apt, and R.~Mancini.
\newblock {Economics of electric energy storage for energy arbitrage and
  regulation in New York}.
\newblock \emph{Energy Policy}, 35\penalty0 (4):\penalty0 2558--2568, 2007.

\bibitem[Wen and David(2000)]{Wen2000}
F.~Wen and A.~K. David.
\newblock {Strategic bidding in competitive electricity markets: A literature
  survey}.
\newblock \emph{Power Engineering Society Summer Meeting, 2000. IEEE},
  4:\penalty0 2168--2173, 2000.

\bibitem[Yang et~al.(2011)Yang, Zhang, Kintner-Meyer, Lu, Choi, Lemmon, and
  Liu]{Yang2011}
Z.~Yang, J.~Zhang, M.~C.~W. Kintner-Meyer, X.~Lu, D.~Choi, J.~P. Lemmon, and
  J.~Liu.
\newblock {Electrochemical energy storage for green grid}.
\newblock \emph{Chemical reviews}, 111\penalty0 (5):\penalty0 3577--613, 2011.

\end{thebibliography}

\end{document}